\documentclass[a4paper,10pt,reqno]{amsart}

\usepackage{latexsym,amssymb}
\usepackage{hyperref}
\usepackage{amsmath,amsthm}
\usepackage{amsfonts}
\usepackage[american]{babel}
\usepackage{mathrsfs}
\usepackage{psfrag}
\usepackage{epsfig}
\usepackage{graphpap,latexsym,epsf}
\usepackage{color}
\usepackage{amssymb,eucal,paralist,color,enumerate}
\usepackage{graphicx}

\newcommand{\ep}{\varepsilon}
\newcommand{\ffi}{\varphi}
\newcommand{\teta}{\vartheta}
\newcommand{\R}{\mathbb{R}}
\newcommand{\N}{\mathbb{N}}

\newcommand{\MM}{\mathbb{M}^{2{\times}2}}
\newcommand{\MMM}{\mathbb{M}^{3{\times}3}}
\newcommand{\Mmn}{\mathbb{M}^{m{\times}n}}

\newcommand{\MNN}{\mathbb{M}^{N{\times}N}}

\newcommand{\Om}{\Omega}
\newcommand{\pa}{\partial}

\newcommand{\na}{\nabla}

\newcommand{\rank}{{\rm rank}\,}
\newcommand{\tr}{{\rm tr}\,}
\newcommand{\divv}{{\rm div}\,}

\mathchardef\emptyset="001F

\numberwithin{equation}{section}
\numberwithin{figure}{section}
\numberwithin{subsection}{section}

\newtheorem{theorem}{Theorem}[section]
\newtheorem{prop}[theorem]{Proposition}

\newtheorem{lemma}[theorem]{Lemma}
\newtheorem{remark}[theorem]{Remark}
\newtheorem{defi}[theorem]{Definition}

\title[Attainment results for nematic elastomers]{Attainment results for nematic elastomers}
\author[V.~Agostiniani]{Virginia Agostiniani}
\address{OxPDE - Mathematical Institute, 
Woodstock Road, Oxford OX2 6GG- UK}
\email{Virginia.Agostiniani@maths.ox.ac.uk}

\author[G.~Dal Maso]{Gianni Dal Maso}
\address{SISSA, via Bonomea 265, 34136 Trieste - Italy}
\email{dalmaso@sissa.it}

\author[A.~DeSimone]{Antonio DeSimone}
\address{SISSA, via Bonomea 265, 34136 Trieste - Italy}
\email{desimone@sissa.it}

\thanks{V.~A. has received funding from the 
European Research Council under the 
European Union's Seventh Framework Programme (FP7/2007-2013) / ERC grant agreement ${\rm n^o}$ 291053. Partial funding has been provided also from the ERC Advanced Grants QuaDynEvoPro, grant agreement ${\rm n^o}$ 290888, and MicroMotility, grant agreement ${\rm n^o}$ 340685, and  by the Italian 
Ministry of Education, University, and
Research through the Project ``Calculus of Variations" (PRIN 2010-11).}

\date{\today}

\begin{document}


\begin{abstract}

We consider a class of non-quasiconvex frame indifferent energy densities which
includes Ogden-type energy densities for nematic elastomers.
For the corresponding geometrically linear problem we provide an
explicit minimizer of the energy functional satisfying a nontrivial boundary condition. 
Other attainment results, both for the nonlinear and the linearized model, are obtained by using the theory of convex integration
introduced by M\"uller and \u Sver\'ak in the context of crystalline solids.  

\end{abstract}

\maketitle

\keywords{Keywords: nematic elastomers, convex integration, solenoidal fields.}

\subjclass{MSC 2010: 74G65, 
74B20,
74B15, 
76A15.} 


\section{Introduction}


Nematic elastomers are rubber-like solids made of a polymer network incorporating nematogenic molecules.
One of the main features of these materials is their ability to accommodate macroscopic deformations at no energy cost. 
Indeed, while the nematic mesogens are randomly oriented at high temperature, 
below a certain transition temperature they align to have 
their long axes roughly parallel and this alignment causes a spontaneous elastic deformation of the underlying polymer network.   
If $n\in S^2$ represents the direction of the nematic alignment,
the gradient of the induced spontaneous deformation is given by 
\begin{equation}\label{Lepn_intro}
L_n^{1/2}:=a^{\frac13}n\otimes n+a^{-\frac16}(I-n\otimes n),
\end{equation}
where $a>1$ is a non-dimensional material parameter. 
Choosing as reference configuration $\Om$ the one the sample would exhibit 
in the high-temperature phase \cite{ferro}, we consider the energy density 
\begin{equation}\label{inco_intro}
W_n(F):=
\sum_{i=1}^N\frac{c_i}{\gamma_i}\left[\tr\left(L_n^{-\frac12}FF^TL_n^{-\frac12}\right)^{\frac{\gamma_i}2}-3\right],
\qquad\det F=1,
\end{equation}
where $F\in\MMM$ is a $3{\times}3$ matrix representing the gradient (at a single macroscopic point) 
of a deformation, which maps the reference configuration into the current configuration. 
Moreover, $\gamma_i$ and $c_i$, for $i=1,...,N$, are material constants such that $\gamma_i\geq2$, $c_i>0$.
Note that in (\ref{inco_intro}) the power $\frac{\gamma_i}2$ refers to the matrix $L_n^{-1/2}FF^TL_n^{-1/2}$.
This is an energy density studied in \cite{AgDe2} and can be considered as an ``Ogden-type''
generalization of the classical ``Neo-Hookean'' expression originally proposed by 
Bladon, Terentjev and Warner \cite{Bla} to model an incompressible nematic elastomer.
This is obtained from (\ref{inco_intro}) by setting $N=1$ and $\gamma_1=2$.
Passing to the energy stored by the system when this is free to adjust $n$ at fixed $F$, we define 
\begin{equation}\label{W_come_min}
W(F):=\min_{n\in S^2}W_n(F),
\qquad\det F=1.
\end{equation}
This energy density is always nonnegative and it vanishes precisely when $FF^T=L_n$,
for some $n\in S^2$.
In other words, by left polar decomposition, the set of wells of $W$ is given by 
\[
\bigcup_{n\in S^2}\Big\{L_n^{\frac12}R\,:\,R\mbox{ is a rotation}\Big\}.
\]
Some experimental tests on samples of nematic elastomers show that these
materials tend to develop microstructures. 
In the mathematical model presented above, 
the formation of microstructures, which heavily influences the macroscopic material response,
is encoded in the energy-wells structure, which makes the density non quasiconvex.
In fact, minimization with respect to $n$ leads to a loss of stability of homogeneously
deformed states with respect to configurations which exhibit shear bands and look like
stripe domains.
Adopting a variational point of view, one is then typically interested in the study
of the free-energy functional $I(y):=\int_{\Om}W(\na y)dx$, with $y:\Om\to\R^3$
a deformation, under the basic assumption that the observed microstructures
correspond to minimizers or almost minimizers of $I$.   
It is also worth mentioning that some appropriate dynamical models for nematic elastomers
and the associated time-dependent evolutions may select, in the limit as the time goes to infinity,
some minimizers of $I$. 
In fact, the study of the existence of exact minimizers attempted in this paper is a natural first step before attempting  the analysis of time-dependent models.

In this paper, we provide some results concerning the existence of minimizers for $I$,
subject to suitable boundary conditions, focussing on
solutions $y$ which minimize the integrand pointwise, that is $W(\na y)=0$ a.e. in $\Om$.
More in general, we deal with energy densities of the form  
\begin{equation}\label{en_W_minimi_intro}
W(F):=
\sum_{i=1}^N\frac{c_i}{\gamma_i}\left[\left(\frac{\lambda_1(F)}{\mathsf e_1}\right)^{\gamma_i}
+\left(\frac{\lambda_2(F)}{\mathsf e_2}\right)^{\gamma_i}
+\left(\frac{\lambda_3(F)}{\mathsf e_3}\right)^{\gamma_i}-3\right],\qquad\det F=1,
\end{equation}
where $0<\lambda_1(F)\leq\lambda_2(F)\leq\lambda_3(F)$ are the ordered singular values of $F$, and
$0<\mathsf e_1\leq\mathsf e_2\leq\mathsf e_3$ are three fixed ordered real numbers such that
$\mathsf e_1\mathsf e_2\mathsf e_3=1$ and $\mathsf e_1<\mathsf e_3$.
The case $\mathsf e_1=\mathsf e_3$ is trivial because in this case 
expression (\ref{en_W_minimi_intro}) corresponds to the classical Neo-Hookean model.  
The Ogden-type energy density obtained by minimizing (\ref{inco_intro}) with respect to $n$ 
is included in (\ref{en_W_minimi_intro}) choosing $\mathsf e_1=\mathsf e_2=a^{-1/6}$ and $\mathsf e_3=a^{1/3}$ 
(see \cite[Proposition 5.1]{AgDe2}).
By using the standard inequality between geometric and arithmetic mean,
it is easy to see that the function $W(F)$ is minimized at the value zero if $F$ is in the set
\begin{equation}\label{wells_W_minimi}
K:=\left\{F\in\MMM:\det F=1\mbox{ and }\lambda_i(F)=\mathsf e_i,\,i=1,2,3\right\}.
\end{equation}

In this paper, we also treat the geometrically linear counterpart of the 
minimization problem associated with the density (\ref{W_come_min}). 
In this case, the problem consists in finding minimizers (which again minimize the integrand pointwise) of the free-energy
functional $\int_{\Om}V(e(u))dx$, where $u:\Om\to\R^3$ is a displacement vector field
subject to suitable boundary conditions and $e(u)$ denotes the symmetric part of $\na u$. 
Here, the energy density $V$ governing the purely mechanical response 
of the system in the small strain limit
is given, up to a multiplicative constant, by
\begin{equation}\label{en_V_dim2}
V(E):=\min_{n\in S^2}|E-U_n|^2,\qquad U_n:=\frac12(3n{\otimes}n-I),
\end{equation}
for every symmetric matrix $E\in\MMM$ such that $\tr E=0$.
The derivation of this expression from (\ref{inco_intro})-(\ref{W_come_min}) 
is recalled in Section \ref{attainment_linear}.
Clearly, we have that $V(E)=0$ if and only if $E=U_n$ for some $n\in S^2$ or,
equivalently, if and only if $E$ is in the set
\begin{equation}\label{nostroK_sym}
\hat K_0:=\left\{E\in\MMM\mbox{ symmetric}:\mu_1(E)=\mu_2(E)=-\frac12,\,\mu_3(E)=1\right\},
\end{equation}
where $\mu_1(E)\leq\mu_2(E)\leq\mu_3(E)$ are the ordered eigenvalues of $E$.

In Theorem \ref{slu_rot} we provide the explicit expression of a solution to the problem
\begin{equation}\label{cond_minimi_intro_22}
V(e(u))=0\quad\mbox{a.e.~in }\Om,\qquad u=w\quad\mbox{on }\pa\Om,
\end{equation}
when $\Om=B(0,r)\times\R$, and 
$w(x_1,x_2,x_3):=(\frac{x_1}4,\frac{x_2}4,-\frac{x_3}2)$. 
Note that the affine extension of $w$ to the interior of $\Om$ is such that 
$e(w)$ is a constant matrix not belonging to the set of minimizers $\hat K_0$.
As a consequence, the chosen boundary datum $w$ is nontrivial in the sense that $V(e(w))$ is a strictly positive constant.
The explicit solution we find, which is of class $W^{1,p}$ for every $1\leq p<\infty$, 
allows us to construct solutions to problem 
(\ref{cond_minimi_intro_22}) (endowed with the same regularity), 
for domains of the form $\omega\times\R$, $\omega$ being an open subset of $\R^2$.
Theorem \ref{slu_rot} shows that, thanks to the symmetries of $\hat K_0$, 
one can exhibit a simple explicit solution.
For general domains such an explicit solution is no longer available and,
just as in the case of solid crystals, many solutions of the
minimization problem exist but they can only be defined through iterative procedures. 

Theorem \ref{teoatteso_nonlinear} states that for every 
function $v:\Om\to\R^3$ which is piecewise affine and Lipschitz, if 
\begin{equation}\label{cond_minimi_intro_1}
\det \na v=1\quad\mbox{a.e.~in }\Om,\qquad
{\rm ess\,inf}_{\Om}\,\lambda_1(\na v)>\mathsf e_1,\qquad
{\rm ess\,sup}_{\Om}\,\lambda_3(\na v)<\mathsf e_3,
\end{equation}
then there exists a Lipschitz function $y:\Om\to\R^3$ such that
\begin{equation}\label{prob_nonlinear_intro_1}
W(\na y)=0\quad\mbox{a.e.~in }\Om,\qquad y=v\quad\mbox{on }\pa\Om.
\end{equation}
The same holds if $v$ is of class $C^{1,\alpha}(\overline{\Om};\R^3)$, for some $0<\alpha<1$,
and satisfies \eqref{cond_minimi_intro_1}.
Moreover, the solution $y$ can be chosen to be arbitrarily close to $v$ in $L^{\infty}$-norm.
This result is an application of the theory developed by M\"uller and \u Sver\'ak in \cite{MuSv2}
where the authors use Gromov's convex integration theory to study the existence of 
solutions of the first order partial differential relation 
\begin{equation}\label{prob_nonlinear_intro_11}
\na y\in\tilde K\quad\mbox{a.e.~in }\Om,\qquad y=v\quad\mbox{on }\pa\Om.
\end{equation} 
Here the set $\tilde K$ is contained in $\{F:M(F)=t\}$, $M(F)$
being a fixed minor of $F$, and $t\neq 0$. The case $M(F)=\det F$ and $t=1$ perfectly applies 
to our minimization problem (\ref{prob_nonlinear_intro_1}),
which can be rewritten as \eqref{prob_nonlinear_intro_11} with $\tilde K=K$.
A crucial step in the theory 
is the construction of a suitable approximation
of $\tilde K$ by means of sets relatively open in $\{F:\det F=1\}$ and satisfying some
technical assumptions (see Definition \ref{inapprox_per_nonlinear}). 
To obtain Theorem \ref{teoatteso_nonlinear}
we provide such an approximation for our set $K$ and apply
the results of \cite{MuSv2} directly.
 
To give a corresponding attainment result in the geometrically linear setting,
we have to consider the case where the set $\tilde K$ appearing in \eqref{prob_nonlinear_intro_11} is contained
in $\{F:\tr F=0\}$. 
The constraint on the determinant is then replaced by a constraint
on the divergence. This case is not explicitly treated in \cite{MuSv2}
and it has been considered in \cite{Mu-Pa2} to study a partial differential relation arising in the study of the
Born-Infeld equations. Moreover, convex integration techniques coupled with divergence constraints 
have been fruitfully employed by De Lellis and Sz{\'e}kelyhidi in the study of the Euler equations (see, e.g., \cite{Cam_Laz}).
In order to be self-contained we state and prove Theorem \ref{generici} 
and Proposition \ref{perHol}, which are a ``linearized'' version of
some of the results in \cite{MuSv2}.
We then apply Theorem \ref{generici} and Proposition \ref{perHol} 
to obtain the result which is described next (Theorem \ref{teoatteso_linear_3dim}).

Consider the small strain energy density $V$ and let us introduce the set
\begin{equation}\label{nostroK}
K_0:=\left\{A\in\MMM:\frac{A+A^T}2\in\hat K_0\right\},
\end{equation}
where $\hat K_0$ is defined in (\ref{nostroK_sym}). We have that $V(\frac{A+A^T}2)=0$ for every $A\in K_0$.

We prove that for every piecewise affine Lipschitz map $w:\Om\to\R^3$ such that 
\begin{equation}\label{cond_minimi_intro_2}
\divv w=0\quad\mbox{a.e.~in }\Om,\qquad
{\rm ess\,inf}_{\Om}\,\mu_1(e(w))>-\frac12,\qquad
{\rm ess\,sup}_{\Om}\,\mu_3(e(w))<1,
\end{equation}
there exists a Lipschitz function $u:\Om\to\R^3$ satisfying
(\ref{cond_minimi_intro_22}). 
The same conclusion holds if $w$ is of class $C^{1,\alpha}(\overline{\Om};\R^3)$, for some $0<\alpha<1$,
and satisfies \eqref{cond_minimi_intro_2}. Moreover, as for the nonlinear case, 
the solution can be chosen to be arbitrarily close to $w$ in $L^{\infty}$-norm. 

To prove this result, we apply Theorem \ref{generici} to the minimization problem (\ref{cond_minimi_intro_22}),
where the condition $V(e(u))=0$ a.e.~in $\Om$ is equivalent to $\na u\in K_0$ a.e.~in $\Om$.
As for the nonlinear case, also in the linearized context the main point consists in
exhibiting a suitable approximation of $K_0$ by means of
sets relatively open in $\{F\in\MMM:\tr F=0\}$ and satisfying some technical assumptions. 
For sake of completeness, we state and prove the
$2$-dimensional version  (Theorem \ref{teoatteso}) of this result, where the condition
(\ref{cond_minimi_intro_2}) is slightly simplified and the energy well structure allows
for more geometrical intuition and a more explicit proof.

The rest of the paper is organized as follows: in Section \ref{attainment_linear} we explain how to construct an
explicit solution to problem (\ref{cond_minimi_intro_22}), and in Section \ref{attainment_nonlinear} 
we state and prove the attainment results obtained by using the
theory of convex integration, for the nonlinear as well as for the geometrically linear case.
Section \ref{dim_teo_gen} is devoted to the proof of the results 
used in Section \ref{attainment_nonlinear}, which are an adaptation
of the approach of \cite{MuSv2} to divergence free vector fields.
 

\section{An explicit solution}\label{attainment_linear}


In this section, we focus on the geometrically linear model.
The set of $N{\times}N$ (real) matrices is denoted by $\MNN$, while
$Sym(N)$ is the subset of symmetric matrices. $\MNN_0$ and $Sym_0(N)$
denote the subsets of matrices in $\MNN$ and $Sym(N)$, respectively, which have null trace. The symbols
$sym A$ and $skw A$ stand for the symmetric and the skew symmetric part of a matrix $A$, respectively.
Given a displacement field $u:\Om\to\R^3$, where $\Om\subset\R^3$ 
is the reference configuration, we use the notation $e(u):=sym(\na u)$.

To derive the linearized version (\ref{en_V_dim2}) of the energy density $W$ defined by (\ref{inco_intro})-(\ref{W_come_min}),
consider the nematic tensor $L_n$ given in (\ref{Lepn_intro}), 
choose $a=(1+\ep)^3$, and 
relabel $L_n$ by $L_{n,\ep}$. By expanding in $\ep$ we have
$$
L_{n,\ep}^{\frac12}=I+\ep U_n+o(\ep),\qquad U_n:=\frac12(3n{\otimes}n-I).
$$  
For sake of completeness, let us derive the geometrically linear model in the
compressible case. We then obtain expression
(\ref{en_V_dim2}) by restricting to null trace matrices.
The following is a natural compressible generalization of expression (\ref{inco_intro}):
\begin{equation}\label{Wnc}
W_n^c(F):=
\sum_{i=1}^N\frac{c_i}{\gamma_i}
\left[(\det F)^{-\frac{\gamma_i}3}\tr\left(L_n^{-\frac12}FF^TL_n^{-\frac12}\right)^{\frac{\gamma_i}2}-3\right]
+W_{vol}(\det F),
\end{equation}
where $F$ is any matrix in $\MMM$ such that $\det F>0$, and $W_{vol}$ is defined as 
$$
W_{vol}(t)=c(t^2-1-2\log t),\qquad t>0,
$$
$c$ being a given positive constant.
As its incompressible version, the energy density (\ref{Wnc}) is always nonnegative and it is
equal to 0 if and only if $FF^T=L_n$ (see \cite{AgDe1}, \cite{AgDe2}, and \cite{Ant} for more details).
We denote by $W_{n,\ep}^c$ the expression obtained from (\ref{Wnc}) replacing $L_n$ by $L_{n,\ep}$.
The linearization of the model is then given by
$$
V_n^c(E):=\lim_{\ep\to0}\frac1{\ep^2}W_{n,\ep}^c(I+\ep E),\qquad E\in Sym(3).
$$
Writing $W_{n,\ep}^c(F)=\tilde W_{n,\ep}^c(FF^T)$ due to frame indifference,
and using the fact that $\tilde W_{n,\ep}^c$ is minimized at $L_{n,\ep}$, it is easy to see that
$$
V_n^c(E)=2D^2\tilde W_{n,0}^c(I)[E-U_n]^2=
\frac12\sum_{i=1}^Nc_i\gamma_i|E-U_n|^2
+\left(-\frac16\sum_{i=1}^Nc_i\gamma_i+2c\right)\tr^2E,
$$
for every $E\in Sym(3)$, where $D^2\tilde W_{n,0}^c(I)[E-U_n]^2$ is the 
second differential of $\tilde W_{n,0}^c$
at $I$ applied to $(E-U_n)$ twice. The purely mechanical
response of the system in the small strain limit is defined by $\min_{n\in S^2}V_n^c(E)$.
Up to a multiplicative constant, this last expression gives precisely the function $V$ defined in (\ref{en_V_dim2}),
for every $E\in Sym_0(3)$. 

Since
\begin{eqnarray*}
V(E)&=&\min_{n\in S^2}\left(|E|^2+|U_n|^2-2E\cdot U_n\right)\\
    &=&|E|^2+\frac32\tr E-3\max_{n\in S^2}(En)\cdot n,
\end{eqnarray*}
if $\mu_1(E)\leq\mu_2(E)\leq\mu_3(E)$ are the ordered eigenvalues of $E$,
then $V(E)$ can be rewritten as
$$
V(E)=\left(\mu_1(E)+\frac12\right)^2+\left(\mu_2(E)+\frac12\right)^2+(\mu_3(E)-1)^2,
$$
and the minimum is attained for $n$ parallel to the eigenvector 
of $E$ corresponding to its maximum eigenvalue. The set of wells of $V$ is the set $\hat K_0$
defined in (\ref{nostroK_sym}). 

In order to conform our language to the one used in the engineering literature, 
we remark that an equivalent way to present the small strain theory is
to say that in the small strain regime $|\na u|=\ep$ we have that, modulo terms of order
higher than two in $\ep$,
$$
W(I+\na u)=\mu\min_{n\in S^2}|e(u)-\ep U_n|,
$$
where $W$ is given by (\ref{inco_intro})-(\ref{W_come_min}) and $\mu$ is a function of the
constants appearing in (\ref{inco_intro}). 
We have in this case that $W(I+\na u)=0$ (modulo terms of order higher than two in $\ep$) if and only if
the eigenvalues of $e(u)$ are $-\frac{\ep}2$, $-\frac{\ep}2$, and $\ep$. 

We consider the problem of finding a minimizer of the functional
$\int_{\Om}V(e(u))dx$, under a prescribed boundary condition. 
We find solutions by solving the following problem:
given a Dirichlet datum $w$, find $u$ such that
$V(e(u))=0$ a.e.~in $\Om$ satisfying $u=w$ on $\pa\Om$. 
Considering the set $K_0$ defined in (\ref{nostroK}), note that 
if $A\in K_0$ and $w(x)=Ax$, then the affine function $x\mapsto Ax$ is trivially a solution.

Denoting by $(x_1,x_2,x_3)$ the coordinates of a point $x\in\R^3$,
we restrict attention to domains of the type
$\Om=\omega\times(0,1)$, $\omega$ being an open subset of $\R^2$, 
and look for solutions $u$ of the form
\begin{equation}\label{type1}
u(x)=(\tilde u(x_1,x_2),0)+w(x),
\end{equation}
where $\tilde u:\omega\to\R^2$ is such that $\tilde u=0$ on $\pa\Om$, and
\begin{equation}\label{type2}
w(x):=\left(\frac{x_1}4,\frac{x_2}4,-\frac{x_3}2\right).
\end{equation}
This choice ensures that $e_{33}(u)$ is constantly equal to $-1/2$ and that the minima of the two-dimensional theory
represent minima of the three-dimensional theory as well (see \cite{CeDe2}, where a similar point of  view is adopted). 
In particular, we have that
\begin{equation}\label{pb2rid}
\int_{\Om}V(e(u))dx=\int_{\omega}V
\left(
\left[
\begin{array}{ccc}
\partial_{x_1}\tilde u_1+\frac14 & \frac{\partial_{x_1}\tilde u_2+\partial_{x_2}\tilde u_1}2 & 0 \\
\frac{\partial_{x_1}\tilde u_2+\partial_{x_2}\tilde u_1}2 & \partial_{x_2}\tilde u_2+\frac14 & 0 \\
0 & 0 & -\frac12
\end{array}
\right]
\right)dx_1dx_2.
\end{equation}
This preliminary remark leads to the following theorem.

\begin{theorem}\label{slu_rot}
Given $r>0$, the function
\begin{equation*}
u(x)=\pm\frac3{4}\left[\log\left(\frac{x_1^2+x_2^2}{r^2}\right)\right](-x_2,x_1,0)+w(x)
\end{equation*}
satisfies
\begin{equation}\label{pb_cilindro}
\left\{
\begin{array}{ll}
V(e(u))=0 & \mbox{in }B(0,r)\times\R,\\
u=w & \mbox{on }\partial(B(0,r)\times\R),
\end{array}
\right.
\end{equation}
and belongs to $W^{1,p}_{loc}(B(0,r)\times\R)$, for every $1\leq p<\infty$.
\end{theorem}

\begin{proof}
The proof is a direct computation.
\end{proof}

It is worth commenting on the steps that led us to the construction of the function $u$
given in Theorem \ref{slu_rot}. To do this, 
let us proceed as anticipated before and look for solutions of type
(\ref{type1})-(\ref{type2}) on $\R\times\omega$. 
We denote by $\tilde u_1$ and $\tilde u_2$ the components of $\tilde u$.
Note that if $E\in Sym_0(3)$ is of the form
\begin{equation}\label{form1}
E=\left[
\begin{array}{ccc}
a+\frac14 & b & 0\\
b & -a+\frac14 & 0\\
0 & 0 & -\frac12
\end{array}
\right],
\end{equation}
then, considering the set $\hat K_0$ of the minimizers of $V$ (see (\ref{nostroK_sym})), it is 
easy to see that
\begin{equation}\label{form2}
V(E)=0\quad\mbox{ if and only if }\quad a^2+b^2=\frac9{16}.
\end{equation}
In view of this and of (\ref{pb2rid}), we look for solutions of the following
nonlinear system of partial differential equations in $\omega$:
\begin{equation}\label{sis_rigido}
\left\{
\begin{array}{l}
\pa_{x_1}\tilde u_1+\pa_{x_2}\tilde u_2=0,\\
(\pa_{x_1}\tilde u_1)^2+\left(\frac{\pa_{x_1}\tilde u_2+\pa_{x_2}\tilde u_1}2\right)^2=\frac9{16}.
\end{array}
\right.
\end{equation}
In order to solve this system, a possible strategy is to choose $\tilde u$ as a $\frac{\pi}2$-(counterclockwise) rotation 
of the gradient of a function $\ffi:\R^2\to\R$, that is 
\begin{equation}\label{ansatz_buona}
\tilde u=(-\pa_{x_2}\ffi,\pa_{x_1}\ffi).
\end{equation}
This gives automatically ${\rm div}\,\tilde u=0$ and the second equation in (\ref{sis_rigido}) becomes
\begin{equation}\label{eq_rigida}
(\pa^2_{x_1x_2}\varphi)^2+\left(\frac{\pa^2_{x_1^2}\varphi-\pa^2_{x_2^2}\varphi}2\right)^2=\frac9{16}.
\end{equation}
This a fully nonlinear second order partial differential equation for which, to the best of
our knowledge, a general theory is not available.
To find a solution to this equation, we look for solutions
of the form $\varphi(x_1,x_2)=\psi(\rho^2)$, where $\rho:=\sqrt{x_1^2+x_2^2}$.
In this case, equation (\ref{eq_rigida}) becomes an ordinary differential
equation in $\rho^2$:
\begin{equation*}
1=(4x_1x_2\psi'')^2+\left(\frac{4x_1^2\psi''-4x_2^2\psi''}2\right)^2=4\rho^4(\psi'')^2,
\end{equation*}
which gives $\varphi(x_1,x_2)=\psi(\rho^2)=\pm\frac38\left(\rho^2\log\rho^2-1\right)+C_1\rho^2+C_2$.
Plugging this expression in (\ref{ansatz_buona}) and imposing $\tilde u=0$ on $\pa B(0,r)$, we obtain
\begin{equation}\label{u2dim_1}
\tilde u(x_1,x_2)=\pm\frac34\log\left(\frac{x_1^2+x_2^2}{r^2}\right)(-x_2,x_1).
\end{equation}
The function 
\begin{equation}\label{u2dim_2}
u(x):=(\tilde u(x_1,x_2),0)+w(x),
\end{equation} 
where $w$ is defined as in (\ref{type2}),
is then a solution of problem (\ref{pb_cilindro}).

We emphasize that the case of $\Om=\omega\times\R$ with $\omega=B(0,r)$
is very special, leading to the explicit solution $u$ defined in (\ref{u2dim_1})-(\ref{u2dim_2}). 
To find a solution when $\omega$ is not a disk,
the strategy is to express $\omega$ as a disjoint union of a sequence of disks and a null set
(see Remark \ref{remark_Vitali}). 
This method
does not provide solutions as explicit as those on $B(0,r)\times\R$.

Observe that the function $\tilde u$ defined in (\ref{u2dim_1}) 
is of class $C(\overline B(0,r);\R^2)$ and that
\begin{equation*}
\na\tilde u(x_1,x_2)=\pm\frac32
\left[
\begin{array}{cc} 
-\frac{x_1x_2}{\rho^2} & -\log\left(\frac{\rho}r\right)-\frac{x_2^2}{\rho^2} \\
\log\left(\frac{\rho}r\right)+\frac{x_1^2}{\rho^2} & \frac{x_1x_2}{\rho^2}
\end{array}
\right], 
\end{equation*}
so that $\na\tilde u\in C^{\infty}(\overline{B(0,r)}\setminus\{0\};\MM)$.
Moreover, $e(\tilde u)\in L^{\infty}(B(0,r);Sym(2))$, whereas
$\na\tilde u$ is unbounded about the origin.
Nevertheless, $\tilde u\in W^{1,p}(B(0,r);\R^2)$, for every $1\leq p<\infty$.

\begin{remark}\label{remark_Vitali}
\rm
If $\omega$ is an arbitrary open subset of $\R^2$,
by Theorem \ref{Vitali} below there exists a countable collection $\{B_i\}$
of disjoint closed disks in $\omega$ such that
$\left|\omega\setminus\bigcup_{i}B_i\right|=0$.
Let $\xi_i\in\R^2$ and $r_i>0$ be the centre and the radius of the ball $B_i$, respectively.
Considering the function $\tilde u$ defined in (\ref{u2dim_1}), 
the function given by
\begin{equation*}
u^{(i)}(\xi):=r_i\tilde u\left(\frac{\xi-\xi_i}{r_i}\right),\qquad\mbox{for every }\xi\in B_i,
\end{equation*}
satisfies (\ref{sis_rigido}) in $B_i$ and $u^{(i)}=0$ on $\pa B_i$.
Now, define
\begin{equation*}
\tilde v:=
\left\{
\begin{array}{ll}
0 & \mbox{on }\displaystyle\omega\setminus\bigcup_{i}B_i,\\
u^{(i)} & \mbox{on }B_i,\mbox{ for every }i.\\
\end{array}
\right.
\end{equation*}
This function is a solution to problem (\ref{sis_rigido}) in $\omega$. To see this, let us introduce the functions
\begin{equation*}
\tilde v^{(k)}:=
\left\{
\begin{array}{ll}
0 & \mbox{on }\displaystyle\omega\setminus\bigcup_{i=1}^kB_i,\\
u^{(i)} & \mbox{on }B_i,\mbox{ for every }i=1,...,k.\\
\end{array}
\right.
\end{equation*}
Extending each $u^{(i)}$ at zero outside $B_i$, we can also write
$\tilde v=\sum_iu^{(i)}$ and $v^{(k)}=\sum_{i=1}^ku^{(i)}$, so that 
\begin{equation}\label{pointwise_minimi}
\tilde v^{(k)}(x)\to\tilde v(x),\quad\mbox{ as }k\to\infty,\qquad\mbox{for every }x\in\omega.
\end{equation}
Since $|e(\tilde v^{(k)})|\leq3/(2\sqrt 2)$ a.e.~in $\omega$, we have that
the sequence $\{\tilde v^{(k)}\}$ is bounded in $W_0^{1,p}(\omega;\R^2)$, for every $1<p<\infty$,
by Korn's inequality. This fact, together with the 
pointwise convergence (\ref{pointwise_minimi})
gives that $\tilde v\in W_0^{1,p}(\omega;\R^2)$, for every $1<p<\infty$. 
Finally, since $\tilde v$ satisfies (\ref{sis_rigido}) a.e.~in each $B_i$ and $\left|\omega\setminus\bigcup_iB_i\right|=0$, we conclude that $\tilde v$ satisfies (\ref{sis_rigido}) a.e.~in $\omega$.
Therefore, we have obtained that the function
\begin{equation*}
v(x):=(\tilde v(x_1,x_2),0)+w(x)
\end{equation*}
is a $W_{loc}^{1,p}(\omega\times\R;\R^3)$ solution, for every $1\leq p<\infty$, of the problem   
\begin{equation*}
\left\{
\begin{array}{ll}
V(e(v))=0 & \mbox{in }\omega\times\R,\\
v=w & \mbox{on }\partial(\omega\times\R).
\end{array}
\right.
\end{equation*}
\end{remark}

We recall the following fundamental corollary of Vitali's Covering Theorem, 
which is also useful in Section \ref{dim_teo_gen}.
We refer the reader to \cite{DaMa} for its proof.

\begin{theorem}[Corollary of Vitali's Covering Theorem]\label{Vitali}
Let $\Om\subseteq\R^N$ be an open set and $G\subseteq\R^N$ a compact set with $|G|>0$.
Let $\mathscr G$ be a family of translated and dilated sets of $G$ such that for almost every
$x\in\Om$ and $\ep>0$ there exists $\hat G\in\mathscr G$ with {\rm diam}$\,\hat G<\ep$ and $x\in\hat G$.
Then, there exists a countable subset $\{G_k\}\subseteq\mathscr G$ such that
\begin{equation*}
\bigcup_kG_k\subseteq\Om,\ \ \ G_k\cap G_h=\emptyset\ \mbox{ for every }k\neq h,\ \ \ 
\left|\Om\setminus\bigcup_kG_k\right|=0.
\end{equation*}
\end{theorem}


\section{Convex integration applied to nematic elastomers}\label{attainment_nonlinear}



In this section $\Om$ is a bounded and Lipschitz domain of $\R^3$.
The following notion is crucial in the sequel.
 
\begin{defi}\label{piecewiseaffine}
A map $y:\Om\to\R^m$ is {\rm piecewise affine} if it is continuous and there exist countably many mutually
disjoint Lipschitz domains $\Om_i\subseteq\Om$ such that
\begin{equation*}
y_{|\Om_i}\quad\mbox{is affine}\qquad\mbox{and}\qquad\left|\Om\setminus\bigcup_i\Om_i\right|=0.
\end{equation*}
\end{defi}
Note that for every piecewise affine function $y$, the pointwise gradient $\na y(x)$ is defined for a.e.~$x$,
but it may happen that $y\notin W^{1,1}$, even when $\na y$ is bounded.
For instance, in dimension one, the \textit{Cantor-Vitali}
function is piecewise affine according to the previous definition. 


\subsection{The nonlinear case}


We consider the following problem: find a minimizer of $\int_{\Om}W(\na y)dx$, where
$W$ is defined in (\ref{en_W_minimi_intro}), under a prescribed boundary condition. 
We obtain a solution of this problem if we solve the following: given a Dirichlet datum $v$, find $y$ such that
$W(\na y)=0$ a.e.~in $\Om$ and satisfying $y=v$ on $\pa\Om$.
To state and then prove the following theorem,
let us introduce the set
$$
\Sigma:=\{F\in\MMM:\det F=1\},
$$
and recall that we denote by $0<\lambda_1(F)\leq\lambda_2(F)\leq\lambda_3(F)$ 
the ordered singular values of $F\in\Sigma$. We also use the notation 
$\Lambda(F):=\{\lambda_1(F),\lambda_2(F),\lambda_3(F)\}$.

\begin{theorem}\label{teoatteso_nonlinear}
Consider a piecewise affine Lipschitz map $v:\Om\to\R^3$ such that 
\begin{equation}\label{cond_nonlinear}
\na v\in\Sigma\quad\mbox{a.e.~in }\Om,\qquad
{\rm ess\,inf}_{\Om}\,\lambda_1(\na v)>\mathsf e_1,\qquad
{\rm ess\,sup}_{\Om}\,\lambda_3(\na v)<\mathsf e_3.
\end{equation}
Then, for every $\ep>0$ there exists $y_{\ep}:\Om\to\R^3$ Lipschitz such that
\begin{equation*}
W(\na y_{\ep})=0\quad\mbox{a.e.~in }\Om,\qquad y_{\ep}=v\quad\mbox{on }\pa\Om,
\end{equation*}
and $||y_{\ep}-v||_{\infty}\leq\ep$.
The same result holds if $v\in C^{1,\alpha}(\overline{\Om};\R^3)$, for some $0<\alpha<1$,
and satisfies (\ref{cond_nonlinear}).
\end{theorem}

Theorem \ref{teoatteso_nonlinear} says that there exists a numerous set of minimizers 
of the energy (at the level zero).
In these circumstances, the study of an appropriate dynamic model as a method to select minimizers,
in the spirit, e.g.,  of \cite{Bal_Hol} and \cite{Fri_McL}, would be of great interest, 
and we hope to address it in future work.

We recall that a set $U\subseteq\Mmn$ is \emph{lamination convex} if
\begin{equation*}
(1-\lambda)A+\lambda B\in U
\end{equation*}
for every $\lambda\in(0,1)$ and every $A$, $B\in U$ such that {\rm rank}$(A-B)=1$.
The \textit{lamination convex hull} $U^{lc}$ is defined as the smallest lamination convex set which contains $U$.
We also recall (see \cite[Proposition 3.1]{MuSv}) that the lamination convex hull of $U$ 
can be obtained by successively adding rank-one segments, that is
\begin{equation}\label{U^lc}
U^{lc}=\bigcup_{k=0}^{\infty}U^{(k)},
\end{equation}
where $U^{(0)}:=U$ and
\begin{equation}\label{lc_recurs}
U^{(k+1)}:=\{(1-\lambda)A+\lambda B\,:\,A,B\in U^{(k)},\ 
0\leq\lambda\leq1,\ {\rm rank}(A-B)=1\}.
\end{equation}

We remark that the constraint $\Sigma$ is stable under 
lamination, that is if $U\subseteq\Sigma$, then $U^{lc}\subseteq\Sigma$. 
Indeed, if $A$, $B\in\Sigma$ are such that $\rank(A-B)=1$, we can write $A=B+a{\otimes}b$ for 
some vectors $a$, $b$. Thus
$$
1=\det(B^{-1}A)=\det[I+(B^{-1}a){\otimes}b]=1+(B^{-1}a)\cdot b,
$$ 
in view of the fact that $\det[(B^{-1}a){\otimes}b]=0$ and ${\rm Cof\,}[(B^{-1}a){\otimes}b]=0$.
Therefore, we have that $(B^{-1}a)\cdot b=0$ and in turn that
$$
\det[\lambda A+(1-\lambda)B]=\det B\det[I+\lambda(B^{-1}a){\otimes}b]=1,
$$
for every $\lambda\in(0,1)$.

To prove Theorem \ref{teoatteso_nonlinear}, we use the following definition.

\begin{defi}\label{inapprox_per_nonlinear}
Consider $K\subseteq\Sigma$. A sequence of sets $\{U_i\}\subseteq\Sigma$,
where $U_i$ is open in $\Sigma$ for every $i$, 
is an {\rm in-approximation} of $K$ if the following three conditions are satisfied.
\begin{itemize}
\item[{\rm (1)}] $U_i\subseteq U_{i+1}^{lc}$,
\item[{\rm (2)}] $\{U_i\}$ is bounded,
\item[{\rm (3)}] for every subsequence $\{U_{i_k}\}$ of $\{U_i\}$, if $F_{i_k}\in U_{i_k}$ 
and $F_{i_k}\to F$ as $k\to\infty$, then $F\in K$.
\end{itemize}
\end{defi}

We remark that in the literature the third condition in the above definition is stated in the slightly different way:
\begin{equation}\label{weakercond}
\mbox{if}\quad F_i\in U_i\quad\mbox{and}\quad F_i\to F\quad\mbox{as}\quad i\to\infty,\quad\mbox{then}\quad F\in K.
\end{equation}
Note that this condition is not inherited by subsequences, therefore it does not imply
condition (3) of Definition \ref{inapprox_per_nonlinear}.
To see this fact, we can consider an example of in-approximation $\{U_i\}$ 
according to Definition \ref{inapprox_per_nonlinear} 
with the additional property that
\begin{equation}
\label{eq:contro0}
K\mbox{ is disjoint from }U:=\bigcup_{i=1}^{\infty}U_i^{lc},
\end{equation}   
as in the proof of Theorem \ref{teoatteso_nonlinear} below.
We then define $V_i:=U_i$ if $i$ is even, and $V_i:=U_i^{lc}$ if $i$ is odd.
It is easy to see that $\{V_i\}$ satisfies properties (1) and (2) of Definition \ref{inapprox_per_nonlinear}.
It satisfies also (\ref{weakercond}) because if $F_i\in V_i$, then in particular $F_{2i}\in U_{2i}$
and therefore property (3) of Definition \ref{inapprox_per_nonlinear} 
for $\{U_i\}$ implies that $F\in K$.
To see that property (3) does not hold for $\{V_i\}$, fix $G\in U_1^{lc}$ and define $G_{2i+1}=G$ for every $i$.
Then $G_{2i+1}\in U_1^{lc}\subseteq U_{2i+1}^{lc}$ and $G_{2i+1}\to G$, but $G\notin K$, in view of (\ref{eq:contro0}).

Property (3) rather than (\ref{weakercond}) is the crucial one 
in the proof of the following result, which is used to prove Theorem \ref{teoatteso_nonlinear}.

\begin{theorem}\label{generici_nonlinear}
Suppose that $K\subseteq\Sigma$ admits an in-approximation $\{U_i\}$
in the sense of Definition \ref{inapprox_per_nonlinear}. 
Suppose that $v:\Om\to\R^3$ is piecewise affine, Lipschitz, and such that
\begin{equation}\label{na_v_in_U1}
\na v\in U_1\quad\mbox{a.e.~in }\Om.
\end{equation}
Then, for every $\ep>0$ there exists a Lipschitz map $y_{\ep}:\Om\to\R^3$ such that
\begin{itemize}
\item[(i)] $\na y_{\ep}\in K$ a.e.~in $\Om$,
\item[(ii)] $y_{\ep}=v$ on $\pa\Om$,
\item[(iii)] $||y_{\ep}-v||_{{\infty}}\leq\ep$.
\end{itemize}
The same result holds if $v\in C^{1,\alpha}(\overline{\Om};\R^3)$, for some $0<\alpha<1$,
and satisfies (\ref{na_v_in_U1}).
\end{theorem}

We refer the reader to \cite{MuSv2} for the proof of this theorem,
whose analogue in the case of the linear constraint $\divv u=0$ (in place of $\det\na y=1$)
is shown in Section \ref{dim_teo_gen}.

\begin{proof}[Proof of Theorem \ref{teoatteso_nonlinear}]
Finding $y:\Om\to\R^3$ such that
$W(\na y)=0$ a.e.~in $\Om$ is equivalent to finding $y$ such that $\na y\in K$ a.e.~in $\Om$,
where $K$, defined as in (\ref{wells_W_minimi}), is the set of the wells of $W$.
Thus, to prove the theorem, 
we can directly apply Theorem \ref{generici_nonlinear} showing that
$K$ admits an in-approximation in the sense of Definition \ref{inapprox_per_nonlinear}. 
By assumption (\ref{cond_nonlinear}), it is possible to construct a strictly decreasing sequence
$\{\eta_i\}_{i\geq1}$ such that
\begin{equation}\label{lin_in_approx_0}
\mathsf e_1<\eta_1<{\rm ess\,inf}_{\Om}\,\lambda_1(\na v),\qquad \eta_i\to\mathsf e_1,\quad\mbox{as }i\to\infty.
\end{equation}
To define a suitable in-approximation $\{U_j\}$ we need to distinguish the following three cases. 

\begin{itemize}

\item[(1)] If $\mathsf e_1=\mathsf e_2<\mathsf e_3$,
we note that, up to a smaller $\eta_1$,
\begin{equation*}
{\rm ess\,sup}_{\Om}\,\lambda_3(\na v)<\frac1{\eta_1^2}<\mathsf e_3,\qquad\frac1{\eta_i^2}\to\frac1{\mathsf e_1^2}=\mathsf e_3,
\quad\mbox{as }i\to\infty. 
\end{equation*}
Hence, defining
\begin{equation*}
U_1:=\left\{F\in\Sigma\,:\,\Lambda(F)\subset\left(\eta_1,\frac1{\eta_1^2}\right)\right\},
\end{equation*}
we have that $\na v\in U_1$ a.e.~in $\Om$, also in view of (\ref{lin_in_approx_0}).
We then define for $i\geq2$
\begin{equation*}
U_i:=\left\{F\in\Sigma\,:\,\lambda_1(F),\lambda_2(F)\in(\eta_i,\eta_{i-1}),\,
\lambda_3(F)\in\left(\frac1{\eta_{i-1}^2},\frac1{\eta_i^2}\right)\right\}.
\end{equation*} 

\item[(2)] If $\mathsf e_1<\mathsf e_2<\mathsf e_3$,
we consider a strictly increasing sequence $\{\teta_i\}\subset(\mathsf e_1,\mathsf e_3)$
such that 
\begin{equation}\label{lin_in_approx_1}
{\rm ess\,sup}_{\Om}\,\lambda_3(\na v)<\teta_1<\mathsf e_3,\qquad\teta_i\to\mathsf e_3,
\quad\mbox{as }i\to\infty, 
\end{equation}
and we define
\begin{equation*}
U_1:=\left\{F\in\Sigma\,:\,\Lambda(F)\subset\left(\eta_1,\teta_1\right)\right\}.
\end{equation*}
By this definition, from (\ref{lin_in_approx_0}) and (\ref{lin_in_approx_1})
we get that $\na v\in U_1$ a.e.~in $\Om$.
Note that if $F\in\Sigma$, then $\lambda_2(F)=(\lambda_1(F)\lambda_3(F))^{-1}$, so that
$1/(\eta_{i-1}\teta_i)<\lambda_2(F)<1/(\eta_i\teta_{i-1})$, if
$\eta_i<\lambda_1(F)<\eta_{i-1}$ and $\teta_{i-1}<\lambda_3(F)<\teta_i$.
Therefore, we define for $i\geq2$
\begin{equation*}
U_i:=\left\{F\in\Sigma\,:\,\lambda_1(F)\in(\eta_i,\eta_{i-1}),\,
\lambda_2(F)\in\left(\frac1{\eta_{i-1}\teta_i},\frac1{\eta_i\teta_{i-1}}\right),\,
\lambda_3(F)\in(\teta_{i-1},\teta_i)\right\}.
\end{equation*} 

\item[(3)] If $\mathsf e_1<\mathsf e_2=\mathsf e_3$, since in this case $1/\sqrt{\eta_i}\to1/\sqrt{\mathsf e_1}=\mathsf e_3$,
we define
\begin{equation*}
U_1:=\left\{F\in\Sigma\,:\,\Lambda(F)\subset\left(\eta_1,\frac1{\sqrt{\eta_1}}\right)\right\},
\end{equation*}
and for $i\geq2$
\begin{equation*}
U_i:=\left\{F\in\Sigma\,:\,\lambda_1(F)\in(\eta_i,\eta_{i-1}),\,
\lambda_2(F),\lambda_3(F)\in\left(\frac1{\sqrt{\eta_{i-1}}},\frac1{\sqrt{\eta_i}}\right)\right\}.
\end{equation*} 
Also in this case, we have that, up to a smaller $\eta_1$, $\na v\in U_1$ a.e.~in $\Om$.

\end{itemize}

It is clear that in each of these cases $U_i$ is open in $\Sigma$ for every $i\geq1$, that 
$\{U_i\}_{i\geq1}$ is bounded, and that 
if $F_i\in U_i$ and $F_i\to F$, then $\Lambda(F)=\{\mathsf e_1,\mathsf e_2,\mathsf e_3\}$.
Now, let us check that $U_i\subseteq U_{i+1}^{lc}$ for every $i\geq1$.
We note that
\begin{equation}\label{inclusione_hull_nonlinear}
\left\{F\in\Sigma\,:\,\Lambda(F)\subset\left(\eta_{i+1},\frac1{\eta_{i+1}^2}\right)\right\}
\subseteq U_{i+1}^{lc}.
\end{equation}
To see this, let us focus on case (1) (in the other cases, inclusion (\ref{inclusione_hull_nonlinear}) can
be proved similarly). For every $\alpha>\eta_{i+1}$ sufficiently close to
$\eta_{i+1}$, we have that 
$\eta_{i+1}<\alpha<\eta_i$ (and $1/\eta_i^2<1/\alpha^2<1/\eta_{i+1}^2$),
so that
\begin{equation*}
\left\{F\in\Sigma:\lambda_1(F)=\lambda_2(F)=\alpha,\,\lambda_3(F)=\frac1{\alpha^2}\right\}\subset U_{i+1}.
\end{equation*}
Thus,
\begin{equation}\label{minimi_nonli_ineq}
\left\{F\in\Sigma:\Lambda(F)\subset\left[\alpha,\frac1{\alpha^2}\right]\right\}
=\left\{F\in\Sigma:\lambda_1(F)=\lambda_2(F)=\alpha,\,\lambda_3(F)=\frac1{\alpha^2}\right\}^{lc}
\subseteq U_{i+1}^{lc},
\end{equation}
where the first equality is guaranteed by Theorem \ref{envelopes_De_Do} below.
Therefore, since (\ref{minimi_nonli_ineq}) is true for every $\alpha>\eta_{i+1}$ sufficiently close to
$\eta_{i+1}$, inclusion (\ref{inclusione_hull_nonlinear}) follows.
The fact that trivially 
$U_i\subseteq\left\{F\in\Sigma\,:\,\Lambda(F)\subset(\eta_{i+1},1/\eta_{i+1}^2)\right\}$
and (\ref{inclusione_hull_nonlinear}) conclude the proof that condition (1) of Definition \ref{inapprox_per_nonlinear} holds and conclude the proof of the theorem. 
\end{proof}

We refer the reader to \cite{DeDo2} for the proof of the following result, which has been used in the proof of Theorem \ref{teoatteso_nonlinear}.

\begin{theorem}\label{envelopes_De_Do}
Let $K$ be given by (\ref{wells_W_minimi}). We have that
\begin{equation*}
K^{lc}=K^{(2)}=\left\{F\in\Sigma\,:\,\Lambda(F)\subset[\mathsf e_1,\mathsf e_3]\right\},
\end{equation*}
where $K^{(2)}$ is the set of second order laminates of $K$.
\end{theorem}


\subsection{The $2$-dimensional geometrically linear case}


As done in Section \ref{attainment_linear}, we restrict attention to displacement vector fields
$u:\Om\to\R^3$ of the form 
\begin{equation*}
u(x_1,x_2,x_3)=(\tilde u(x_1,x_2),0)+w(x_1,x_2,x_3),
\end{equation*}
with $\Om=\omega\times\R$, $\tilde u=(\tilde u_1,\tilde u_2)$, and $w$ given by (\ref{type2}).
As we have already seen, displacements of this form are solution to the problem
\begin{equation*}
V(e(u))=0\quad\mbox{ a.e.~in }\ \Om,\qquad \tilde u=\tilde v\ \mbox{ on }\ \partial\omega,
\end{equation*}
where $V$ is defined in (\ref{en_V_dim2}), if and only if $\tilde u:\omega\to\R^2$ satisfies
\begin{equation}\label{eq:pb_rigido}
\left\{
\begin{array}{l}
\pa_{x_1}\tilde u_1+\pa_{x_2}\tilde u_2=0\quad\mbox{in }\ \omega,\\
(\pa_{x_1}\tilde u_1)^2+\left(\frac{\pa_{x_1}\tilde u_2+\pa_{x_2}\tilde u_1}2\right)^2=\frac9{16}\quad\mbox{in }\ \omega,\\
\tilde u=\tilde v\quad\mbox{on }\ \partial\omega.
\end{array}
\right.
\end{equation}
In Section \ref{attainment_linear} we have provided an explicit solution to
problem (\ref{eq:pb_rigido}) in the case where $\tilde v=0$. Our aim is now to
allow for more general boundary conditions relying on the same techniques 
used for the nonlinear setting in the previous subsection.
Let us do a preliminary observation which is true also in the $3$-dimensional case.
Referring to (\ref{form1})-(\ref{form2}), note first that for any matrix $E\in Sym_0(2)$ represented as 
$
\tilde E=\left[
\begin{array}{cc}
a & b\\
b & -a
\end{array}
\right]
$,
the condition $a^2+b^2=9/16$ is equivalent to 
$\mu_1(\tilde E)=-3/4$, where $\mu_1(\tilde E)$ is the smallest eigenvalue of $\tilde E$.
Defining 
\begin{equation}\label{set_tilde_U}
\tilde U_n:=\frac34(2n\otimes n-I),
\qquad\tilde{\mathcal U}:=\{\tilde U_n:n\in S^1\}=\left\{\tilde E\in Sym_0(2):|\tilde E|=\frac3{2\sqrt 2}\right\},
\end{equation}
and 
\[
\tilde V:Sym_0(2)\to\R,\qquad\quad \tilde V(\tilde E):=
\min_{U\in\tilde{\mathcal U}}|\tilde E-U|^2,
\]
we have that 
$\tilde V(\tilde E)=0$ if and only if $\tilde E\in\tilde{\mathcal U}$.
The results of \cite{Ce} show that the relaxation of the functional 
$\int_{\omega}\tilde V(e(\tilde u))dx$ in the weak sequential topology
of $W^{1,2}(\omega,\R^2)$ is given by $\int_{\omega}\tilde V^{qce}(e(\tilde u))dx$ 
(for every $\tilde u$ such that $\divv\tilde u=0$),
where $\tilde V^{qce}$ is the \textit{quasiconvex envelope on linear strains} of $V$ 
(see \cite{Zhang} for a definition). This is given by 
\begin{equation*}
\tilde V^{qce}(\tilde E)=
\min_{Q\in\tilde{\mathcal Q}}|\tilde E-Q|^2,\qquad\tilde E\in Sym_0(2),
\end{equation*}
with
\begin{equation*}
\tilde{\mathcal Q}:=
\left\{\tilde E\in Sym_0(2):\mu_1(\tilde E)\geq-\frac34\right\}
=\left\{\tilde E\in Sym_0(2):|\tilde E|\leq\frac3{2\sqrt 2}\right\}.
\end{equation*}
In particular, if $\tilde v\in W^{1,2}(\omega;\R^2)$ is such that 
\begin{equation}\label{sugg1}
\divv\tilde v=0\qquad\mbox{and}\qquad|e(\tilde v)|\leq\frac3{2\sqrt 2}\qquad\mbox{a.e.~in }\Om, 
\end{equation}
then
\begin{equation}\label{sugg2}
\inf_{\tilde u\in\tilde v+W_0^{1,2}}\int_{\omega}\tilde V(e(\tilde u))dx
=\min_{\tilde u\in\tilde v+W_0^{1,2}}\int_{\omega}\tilde V^{qce}(e(\tilde u))dx=0.
\end{equation}
The following theorem tells us that if the second condition in (\ref{sugg1}) is a bit stronger,
then there exist minimizers of the unrelaxed functional too. 
In the remaining part of this subsection we use the notation $u$ and $v$ instead of
the notation $\tilde u$ and $\tilde v$
for  2-dimensional displacement vector fields.

\begin{theorem}\label{teoatteso}
Let $v:\omega\to\R^2$ a piecewise affine Lipschitz map such that 
\begin{equation}\label{cond_teo_linear}
\na v\in\MM_0\quad\mbox{a.e.~in }\ \omega,\qquad
{\rm ess\,sup}_{\omega}|e(v)|<\frac3{2\sqrt 2}.
\end{equation}
Then, for every $\ep>0$ there exists $u_{\ep}\in W^{1,\infty}(\omega;\R^2)$ such that
\begin{equation*}
\tilde V(e(u_{\ep}))=0\quad\mbox{a.e.~in }\ \omega,\qquad u_{\ep}=v\quad\mbox{on }\ \pa\omega,
\end{equation*}
and $||u_{\ep}-v||_{\infty}\leq\ep$.
The same result holds if $v\in C^{1,\alpha}(\overline{\omega};\R^2)$, for some $0<\alpha<1$,
and satisfies (\ref{cond_teo_linear}).
\end{theorem}

Condition (\ref{sugg1}) and equality (\ref{sugg2}) lead to suppose that the result of
Theorem \ref{teoatteso} can be obtained even with $|e(v)|\leq(3/2\sqrt2)$ a.e.~in $\omega$.
Nevertheless, the proof of Theorem \ref{teoatteso} strongly relies on the open 
relation appearing in (\ref{cond_teo_linear}).
To prove Theorem \ref{teoatteso} we apply Theorem \ref{generici} (restricted to the case $N=2$).
In order to do this, we have to exhibit an in-approximation 
in the sense of Definition \ref{inapprox_per_nonlinear} (with $\MNN_0$ in place of $\Sigma$) of the set
\begin{equation*}
K_0:=\{A\in\MM_0:sym A\in\tilde{\mathcal U}\},
\end{equation*}
where $\tilde{\mathcal U}$ is defined in (\ref{set_tilde_U}).
We then use Proposition \ref{perHol} to extend 
the result to the case  $v\in C^{1,\alpha}(\overline{\omega};\R^2)$.
Representing every $A\in\MM_0$ as
\begin{equation}\label{rappre_matrici_0}
A=sym A+skw A=
\left[\begin{array}{cc}
a_1 & a_2 \\
a_2 & -a_1
\end{array}
\right]+
\left[\begin{array}{cc}
0 & a_3 \\
-a_3 & 0
\end{array}
\right],
\end{equation}
and denoting $a:=(a_1,a_2)$, it is easy to verify that the condition ${\rm rank\,}(A-B)=1$ is equivalent to
\begin{equation}\label{rank1}
(a_3-b_3)^2=|a-b|^2,\qquad\mbox{for every  }A,B\in\MM_0,
\end{equation}
and the set $K_0$ has the equivalent expression
\begin{equation*}
K_0=\left\{A\in\MM_0\,:\,|a|=\frac34\right\}.
\end{equation*}
Since an in-approximation has to be bounded, 
for the following proof it is useful to introduce the sets
\begin{equation*}
K_0^m:=\{A\in K_0\,:\,|a_3|\leq m\}
\end{equation*}
and
\begin{equation*}
\mathscr C_m:=\left\{A\in\MM_0\,:\,|a_3|\in\left(m-\frac34,m\right)
\ \ \mbox{and }\ \ |a|<|a_3|-m+\frac34\right\},
\end{equation*}
for some constant $m>\frac34$.
In Figure \ref{fig:barattolo} the set $K_0^m$ ios the region inside the big cylinder and the 
set $\mathscr C_m$ is the region bounded by the two cones.

\begin{proof}[Proof of Theorem \ref{teoatteso}]
Suppose $v:\Om\to\R^2$ to be piecewise affine and Lipschitz. Since by hypothesis
\begin{equation*}
M:={\rm ess\,sup}_{\Om}\frac{|e(v)|}{\sqrt2}<\frac34,
\end{equation*}
by choosing $\max\{3/8,M\}<r_0<3/4$, we have that
\begin{equation}\label{cond_teo_generici}
\na v\in U_1:=\left\{A\in\MM_0\,:\,|a|<r_0\,,\,|a_3|<m\right\}\setminus\overline{\mathscr C_m}
             \quad\mbox{a.e.~in }\Om,
\end{equation}
for some $m>3/4$.
In order to use Theorem \ref{generici}, we construct a suitable in-approximation of   
$K_0^m$ starting from $U_1$. 
We consider a strictly increasing sequence $\{r_i\}_{i\geq1}\subset\R$ 
such that $r_1>r_0$ and $r_i\to(3/4)^-$ as $i\to\infty$, and define
\begin{equation}\label{nostro_U_1}
U_i:=\{A\in\MM_0\,:\,r_{i-1}<|a|<r_i\,,\,|a_3|<m\}\setminus\overline{\mathscr C_m},\qquad i\geq2.
\end{equation}
See Figure \ref{fig:barattolo} for a sketch of the sets $U_i$.
Observe that $\{U_i\}$ is a bounded sequence of sets open in $\MM_0$. 
Also, it is clear from the geometry of these sets that whenever $F_i\in U_i$ and $F_i\to F\in\MM$ as $i\to\infty$,
then $F\in K_0^m$.
It remains to check that the first condition of Definition \ref{inapprox_per_nonlinear} hold.
Consider $C\in U_i$ and suppose for simplicity that $0\leq c_3<m$
(the case $-m<c_3<0$ can be treated in a similar way). 
In particular, we have that $|c|<r_i$ and, if $c_3>m-1$, then  
\begin{equation}\label{lemmaKm11}
|c|\geq c_3-m+\frac34,
\end{equation}
by definition of $\mathscr C_m$.
We have to prove that there exist $A$, $B\in U_{i+1}$ such that
\begin{equation}\label{lemmaKm1}
{\rm rank}(A-B)=1\ \ \mbox{ and }\ \ C=(1-\lambda)A+\lambda B,\quad
\mbox{ for some }0<\lambda<1,
\end{equation}
so that $U_i\subseteq U_{i+1}^{(1)}$ (where $U_{i+1}^{(1)}$ is the set of first order laminates of $U_{i+1}$)
and therefore $U_i\subseteq U_{i+1}^{lc}$, as required. 
We fix $\tilde r\in(r_i,r_{i+1})$ and choose
\begin{equation}\label{lemmaKm12}
a:=\frac{\tilde rc}{|c|},\quad\quad\quad b:=-\frac{\tilde rc}{|c|},
\end{equation}
so that $|a-b|=2\tilde r$ and 
\begin{equation*}
c=(1-\lambda)a+\lambda b,\qquad\mbox{ with }\ \ \lambda:=\frac{\tilde r-|c|}{2\tilde r}.
\end{equation*}
With this choice we have that the second condition in (\ref{lemmaKm1}) 
is realized if and only if
\begin{equation}\label{lemmaKm2}
c_3=(1-\lambda)a_3+\lambda b_3.
\end{equation}
Then, choosing  $a_3=b_3+2\tilde r$, the second condition in (\ref{lemmaKm1}),
which is equivalent to (\ref{rank1}), is satisfied, and (\ref{lemmaKm2}) is equivalent to
\begin{equation}\label{lemmaKm3}
b_3=c_3-|c|-\tilde r,\qquad\quad a_3=c_3-|c|+\tilde r. 
\end{equation}
We have now to check that $A\in U_{i+1}$ 
(the fact that $B\in U_{i+1}$ can be verified equivalently).
The property $r_i<|a|<r_{i+1}$ comes from (\ref{lemmaKm12}) 
and from the choice of $\tilde r$, and $|a_3|<m$ follows from (\ref{lemmaKm3}).
Indeed, (\ref{lemmaKm3}) trivially gives $a_3>0$ (recall that we are supposing $c_3\in[0,m)$),
while $a_3=c_3-|c|+\tilde r<m$ is trivially true if $c_3\in[0,m-3/4]$ and follows from
(\ref{lemmaKm11}) if $c_3\in(m-3/4,m)$.
Now, suppose that $a_3\in(m-3/4,m)$. We have to verify that 
\begin{equation*}
|a|\geq a_3-m+\frac34.
\end{equation*}
But this is equivalent to (\ref{lemmaKm11}), which is true if $c_3\in(m-3/4,m)$
and trivially true if $c_3\in[0,m-3/4]$. 
This concludes the verification of $A\in U_{i+1}$
and the proof of the fact that $U_i\subseteq U_{i+1}^{(1)}$.
Thus, we have constructed an in-approximation $\{U_i\}$ of $K_0^m\subseteq K_0$. Moreover, from (\ref{cond_teo_generici}), $\na v\in U_1$ a.e.~in $\Om$. We can now apply Theorem \ref{generici} to obtain the first part of the theorem. It remains to consider
the case where $v\in C^{1,\alpha}(\overline{\Om};\R^2)$ (and satisfies (\ref{cond_teo_linear})). 
Proposition \ref{perHol} ensures the existence of a piecewise affine Lipschitz function $v_{\delta}:\Om\to\R^2$ such that
$\divv v_{\delta}=0$ a.e.~in $\Om$, $||v_{\delta}-v||_{W^{1,\infty}}\leq\delta$, and
$v_{\delta}=v$ on $\pa\Om$. If $\delta$ is sufficiently small, 
we have that $\na v_{\delta}\in U_1$ a.e.~in $\Om$,
where $U_1$ is defined in (\ref{nostro_U_1}), and we can proceed as in the first part of the proof.
\end{proof}

\begin{figure}[htbp]
\begin{center}
\includegraphics[width=3.5cm]{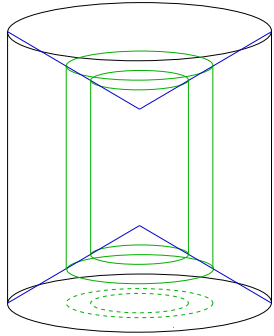}
\end{center}
\caption{Illustration of the sets $K_0^m$, $\mathscr C_m$ and $U_i$ appearing in the proof of 
Theorem \ref{teoatteso}, in the $(a_1,a_2,a_3)$-space.
$U_i$ is the region between the two small cylinders, coloured in green.} 
\label{fig:barattolo}
\end{figure}


\subsection{The $3$-dimensional geometrically linear case}


In this section we consider the $3$-dimensional geometrically linear model and we deal with
the energy density $V$ given by (\ref{en_V_dim2}).
Recall that we denote by 
$\mu_1(E)\leq\mu_2(E)\leq\mu_3(E)$ the ordered eigenvalues of a matrix $E\in Sym(3)$.
We have the following theorem.

\begin{theorem}\label{teoatteso_linear_3dim}
Consider a piecewise affine Lipschitz map $w:\Om\to\R^3$ such that 
\begin{equation}\label{cond_linear_3dim}
\na w\in\MMM_0\quad\mbox{a.e.~in }\ \Om,\qquad
{\rm ess\,inf}_{\Om}\,\mu_1(e(w))>-\frac12,\qquad
{\rm ess\,sup}_{\Om}\,\mu_3(e(w))<1.
\end{equation}
Then, for every $\ep>0$ there exists $u_{\ep}:\Om\to\R^3$ Lipschitz such that
\begin{equation}\label{pb_lin_3dim}
V(e(u_{\ep}))=0\quad\mbox{a.e.~in }\ \Om,\qquad u_{\ep}=w\quad\mbox{on }\ \pa\Om,
\end{equation}
and $||u_{\ep}-w||_{\infty}\leq\ep$.
The same result holds if $w\in C^{1,\alpha}(\overline{\Om};\R^3)$, for some $0<\alpha<1$,
and satisfies (\ref{cond_linear_3dim}).
\end{theorem}

Recall that  the set 
\begin{equation}\label{K_0_3_dim}
K_0:=\left\{A\in\MMM_0:\mu_1(sym A)=\mu_2(sym A)=-\frac12,\,\mu_3(sym A)=1\right\}
\end{equation}
is the set which minimizes $V$ at the level zero, so that
problem (\ref{pb_lin_3dim}) can be rewritten in the equivalent form
\begin{equation*}
\na u_{\ep}\in K_0\quad\mbox{a.e.~in }\ \Om,\qquad u_{\ep}=w\quad\mbox{on }\ \pa\Om.
\end{equation*}
Therefore, in order to prove Theorem \ref{teoatteso_linear_3dim} we can use Theorem \ref{generici}
restricted to $N=3$ if we exhibit an in-approximation $\{U_i\}$ of the set $K_0$ such that 
\begin{equation}\label{cond_thm_linear_3_dim}
\na w\in U_1\quad\mbox{ a.e.~in }\Om.
\end{equation}
For positive constants $\alpha$ and $m$, 
the construction of $\{U_i\}$ in the following proof hinges on the sets
\begin{equation}
\mathcal K_{\alpha,m}:=\{A\in\MMM_0:\mu_1(sym A)=\mu_2(sym A)=-\alpha,\,
\mu_3(sym A)=2\alpha,\,|skw A|<m\}.\label{set1}
\end{equation}
Analogously to the $2$-dimensional case, the restriction $|skw A|<m$ in the above definition is related to the fact that the sequence $\{U_i\}$ has to be bounded.

\begin{proof}
To define a suitable in-approximation of the set $K_0$ defined in (\ref{K_0_3_dim}),
let us first introduce a strictly increasing sequence $\{r_i\}_{i\geq1}$ of positive numbers 
such that $r_i\to(1/2)^-$, as $i\to\infty$, and satisfying
\begin{equation}\label{prop_seq_r_i}
\sum_{i=2}^{\infty}\sqrt{r_{i+1}-r_{i-1}}<\infty.
\end{equation}
It is easy to check that such a sequence exists: it is enough to consider, e.g.,
$\tilde r_i=\sum_{j=1}^i1/j^4$ and $r_i:=\tilde r_i-C+1/2$ where $C:=\sum_{j=1}^{\infty}1/j^4$.
We then choose $\{m_i\}_{i\geq1}$ to be a bounded sequence of positive numbers
such that
\begin{equation}\label{3dim_stimetta0}
m_{i+1}>m_i+4\sqrt{r_{i+1}-r_{i-1}},\qquad\mbox{for every }i\geq1.
\end{equation}
Note that $\{m_i\}$ can be chosen to be bounded in view of (\ref{prop_seq_r_i}).
We now define the sets $U_i$ of the in-approximation for $i\geq2$ and define
$U_1$ later on.
Setting
\begin{multline}\label{def_U_i_3dim}
U_i:=\{A\in\MMM_0:\mu_1(sym A),\mu_2(sym A)\in(-r_i,-r_{i-1}),\\
\mu_3(sym A)\in(2r_{i-1},2r_i),\,|skw A|<m_i\},
\end{multline}
it is clear  that the sets  $U_i$ are open in $\MMM_0$ and equibounded.
Also, if $\{F_i\}$ is a sequence of matrices such that $F_i\in U_i$ for every $i$
and $F_i\to F$, then in particular  $F\in K_0$, because $[-r_i,-r_{i-1}]\to\{-1/2\}$
and $[2r_{i-1},2r_i]\to\{1\}$. To show that $U_i\subseteq U_{i+1}^{lc}$
we show next that $U_i\subseteq U_{i+1}^{(2)}$, where
$U_{i+1}^{(2)}$ is the set of the second order laminates of $U_{i+1}$.
We have the following claim.

\noindent\emph{Claim.}
Setting
\begin{equation*}
\alpha_i:=\frac{r_i+r_{i+1}}2,
\end{equation*}
then $U_i\subseteq\mathcal K_{\alpha_i,m_{i+1}}^{(2)}$,
where $\mathcal K_{\alpha_i,m_{i+1}}$ is given by
(\ref{set1}) with $\alpha_i$ and $m_{i+1}$ in place of $\alpha$ and $m$, respectively.

Note that, once we have proved the claim, the fact that $U_i\subseteq U_{i+1}^{(2)}$
is straightforward, because $\mathcal K_{\alpha_i,m_{i+1}}\subseteq U_{i+1}$ and
in turn $U_i\subseteq\mathcal K_{\alpha_i,m_{i+1}}^{(2)}\subseteq U_{i+1}^{(2)}$.
This concludes the proof of the fact that $\{U_i\}_{i\geq2}$ is an in-approximation.
Once proved the claim we are then left to choose $U_1$ in such a way
that $\{U_i\}_{i\geq1}$ is still an in-approximation and condition 
(\ref{cond_thm_linear_3_dim}) is satisfied.

\noindent\emph{Proof of the claim.}
Fixed $A\in U_i$, we can assume without restrictions that 
$sym A$ is the diagonal matrix ${\rm diag\,}(\mu_2(sym A),\mu_1(sym A),\mu_3(sym A))$. 
Proceeding as in \cite[proof of Corollary 2]{Ce}, let us set
\begin{equation*}
B^+:=A^++skw A,\qquad\quad B^-:=A^-+skw A,
\end{equation*}
with
\begin{equation*}
A^{\pm}:=\left[
\begin{array}{ccc}
\mu_2(sym A) & 0 & 0 \\
0 & \mu_1(sym A) & \pm2\delta \\
0 & 0 & \mu_3(sym A)
\end{array}
\right],
\end{equation*}
so that
\begin{equation}\label{3dim_stimetta01}
A=\frac12 B^++\frac12 B^-.
\end{equation}
Choosing
\begin{equation*}
\delta:=\sqrt{(\alpha_i+\mu_1(sym A))(\alpha_i+\mu_3(sym A))},
\end{equation*}
we obtain that $\delta$ is well-defined and positive because $\mu_1(sym A)>-r_i>-\alpha_i$,
and we get
\begin{equation}\label{3dim_stimetta001}
\mu_1(sym B^{\pm})=-\alpha_i,\quad 
\mu_2(sym B^{\pm})=\mu_2(sym A),\quad
\mu_3(sym B^{\pm})=\alpha_i-\mu_2(sym A).
\end{equation}
Also, note that 
\begin{equation}\label{3dim_stimetta1}
|skw B^{\pm}|\leq|skw A^{\pm}|+|skw A|<\sqrt2\delta+m_i,
\end{equation}
and that, since $\alpha_i<1/2$, $\mu_3(sym A)<1$, and
$-r_{i+1}<-\alpha_i<-r_i<\mu_1(sym A)<-r_{i-1}$, then
\begin{equation}\label{3dim_stimetta2}
\delta<\sqrt{\frac32(\alpha_i+\mu_1(sym A))}<\sqrt{\frac32(r_{i+1}-r_{i-1})}.
\end{equation}
Estimates (\ref{3dim_stimetta1}) and (\ref{3dim_stimetta2}) give
\begin{equation}\label{3dim_stimetta21}
|sym B^{\pm}|<\sqrt{3(r_{i+1}-r_{i-1})}+m_i.
\end{equation}
Now we want to show that, given any matrix $B\in\MMM_0$ such that
\begin{equation}\label{3dim_stimetta3}
\mu_1(sym B)=-\alpha_i,\quad\mu_2(sym B)\in(-r_i,-r_{i-1}),\quad\mu_3(sym B)<2\alpha_i,
\end{equation}
and satisfying
\begin{equation}\label{3dim_stimetta31}
\quad|skw B|<\sqrt{3(r_{i+1}-r_{i-1})}+m_i,
\end{equation}
then 
\begin{equation*}
B=\frac12 C_++\frac12 C_-,\quad\mbox{for some }\ C_+,C_-\in\mathcal K_{\alpha_i,m_{i+1}}
\ \mbox{ such that }\ \rank(C_+-C_-)=1.
\end{equation*}
To see this, let us suppose that $sym B$ is in diagonal form 
\[
sym B={\rm diag\,}(-\alpha_i,\mu_2(sym B),\mu_3(sym B)),
\] 
and, following \cite[proof of Proposition 4]{Ce}, let us write
\begin{equation}
B=\frac12 C_++\frac12 C_-,
\end{equation}
where
\begin{equation*}
C_+:=D_++skw B,\qquad\quad C_-:=D_-+skw B,
\end{equation*}
and
\begin{equation*}
D_{\pm}:=\left[
\begin{array}{ccc}
-\alpha_i & 0 & 0 \\
0 & \mu_2(sym B) & \pm2\ep \\
0 & 0 & \mu_3(sym B)
\end{array}
\right].
\end{equation*}
Choosing
\begin{equation*}
\ep:=\sqrt{(-2\alpha_i+\mu_2(sym B))(-2\alpha_i+\mu_3(sym B))},
\end{equation*}
we obtain that $\ep$ is well-defined and positive because $\mu_3(sym B)<2\alpha_i$,
and we get
\begin{equation}\label{3dim_stimetta4}
\mu_1(sym C_{\pm})=-\alpha_i,\quad 
\mu_2(sym C_{\pm})=-\alpha_i,\quad
\mu_3(sym C_{\pm})=2\alpha_i.
\end{equation}
Moreover, from (\ref{3dim_stimetta3})-(\ref{3dim_stimetta31}) 
and from the fact that $B\in\MMM_0$ we obtain
\begin{equation*}
\ep=\sqrt{(2\alpha_i-\mu_2(sym B))(\alpha_i-\mu_2(sym B))}<
\sqrt{\frac32(r_{i+1}-r_{i-1})},
\end{equation*}
and in turn
\begin{align}\label{3dim_stimetta5}
|skw C_{\pm}|\leq|skw D_{\pm}|+|skw B|&
                                    <\sqrt2\ep+\sqrt{3(r_{i+1}-r_{i-1})}+m_i\nonumber\\
                 & <2\sqrt{3(r_{i+1}-r_{i-1})}+m_i.
\end{align}
Equations (\ref{3dim_stimetta0}), (\ref{3dim_stimetta4}) and (\ref{3dim_stimetta5})
imply that $C_{\pm}\in\mathcal K_{\alpha_i,m_{i+1}}$. The fact 
that $\rank(C_+-C_-)=1$ comes from the construction.

Finally, going back to (\ref{3dim_stimetta01}) and noting from (\ref{3dim_stimetta001}) 
and (\ref{3dim_stimetta21}) that  $B^+$ and $B^-$ 
satisfy (\ref{3dim_stimetta3})-(\ref{3dim_stimetta31}), 
we have that (\ref{3dim_stimetta01}) holds with
\begin{equation}\label{3dim_stimetta6}
B^+=\frac12 C^+_++\frac12 C^+_-,\qquad\quad
B^-=\frac12 C^-_++\frac12 C^-_-,
\end{equation}  
for some $C^+_+$, $C^+_-$, $C^-_+$, $C^-_-\in\mathcal K_{\alpha_i,m_{i+1}}$
such that $\rank(C^+_+-C^+_-)=\rank(C^-_+-C^-_-)=1$. Since from the
construction we have also that $\rank(B^+-B^-)=1$, equations
(\ref{3dim_stimetta01}) and (\ref{3dim_stimetta6}) give that
$A\in\mathcal K_{\alpha_i,m_{i+1}}^{(2)}$. This concludes the proof of the claim.

Now, let us choose the first elements of the sequences $\{r_i\}_{i\geq1}$
and $\{m_i\}_{i\geq1}$ such that
\begin{equation*}
{\rm ess\,inf}_{\Om}\,\mu_1(e(w))>-r_1>-\frac12,\quad
{\rm ess\,sup}_{\Om}\,\mu_3(e(w))<2r_1<1,\quad
m_1>\|\na w\|_{\infty}.
\end{equation*}
By this choice we have that condition (\ref{cond_thm_linear_3_dim}) is satisfied with
\begin{equation*}
U_1:=\{A\in\MMM_0:-r_1<\mu_1(sym A),\,\mu_3(sym A)<2r_1,\,|skw A|<m_1\},
\end{equation*}
which is a set open in $\MMM_0$.
Taking $\alpha_1:=(r_1+r_2)/2$, proceeding as in the proof of the claim gives
that $U_1\in\mathcal K_{\alpha_1,m_2}^{(2)}$ and in turn
$U_1\in U_2^{(2)}$, being $\mathcal K_{\alpha_1,m_2}\subseteq U_2$.
\end{proof}


\section{Appendix}\label{dim_teo_gen}


In this section we prove the following Theorem \ref{generici} and Proposition \ref{perHol}, 
adapting the procedure used in 
\cite{MuSv2} to the linear constraint $\divv u=0$.
The set $\Om$ is a bounded and Lipschitz domain of $\R^N$.
We denote by $[A,B]$ the segment between the matrices $A$ and $B$.
We use the symbols $\|\cdot\|_{\infty}$ and $\|\cdot\|_{1,\infty}$ for the $L^{\infty}$- and the 
$W^{1,\infty}$-norm, respectively. When we want to indicate the domain explicitly, we write
$\|u\|_{L^{\infty}(\Lambda;\R^m)}$ or $\|u\|_{W^{1,\infty}(\Lambda;\R^m)}$, for $u:\Lambda\to\R^m$.

\begin{theorem}\label{generici}
Suppose that $K_0\subseteq\MNN_0$ admits an in-approximation $\{U_i\}$
in the sense of Definition \ref{inapprox_per_nonlinear} with $\Sigma$ replaced by $\MNN_0$. 
Suppose that $v:\Om\to\R^N$ is piecewise affine, Lipschitz, and such that
\begin{equation}\label{cond_in_U_1}
\na v\in U_1\ \ \mbox{a.e.~in }\Om.
\end{equation}
Then, for every $\ep>0$ there exists a Lipschitz map $u_{\ep}:\Om\to\R^n$ such that
\begin{itemize}
\item[(i)] $\na u_{\ep}\in K_0$ a.e.~in $\Om$,
\item[(ii)] $u_{\ep}=v$ on $\pa\Om$,
\item[(iii)] $||u_{\ep}-v||_{\infty}\leq\ep$.
\end{itemize}
\end{theorem}

The proof of this theorem is the last step of an
approximation process which passes through some preliminary results: Lemma \ref{Pompe},
Lemma \ref{4.1MuSv}, and Theorem \ref{aperti}.
In Lemma \ref{Pompe} the following problem is considered:
given two rank-one connected matrices $A$ and $B$ and given $C=(1-\lambda)A+\lambda B$ for some $\lambda\in(0,1)$, 
we construct a map $u$ which satisfies the constraint $\divv u=0$ and the boundary condition
$u(x)=Cx$, and whose gradient lies in a sufficiently small neighbourhood of $[A,B]$.
The next step consists in considering $U$ relatively open in $\MNN_0$.
Lemma \ref{4.1MuSv} states that for every affine boundary data $x\mapsto Cx$ with $C\in U^{lc}$,
there exists a piecewise affine and Lipschitz map $u$ whose gradient   
is always in $U^{lc}$ and is such that the set where $\na u\notin U$ is very small.
Then, by the same iterative method used in the proof
of Lemma \ref{4.1MuSv} it is possible to remove step by step the set where $\na u\notin U$ and
allow for boundary data $v$ such that $\na v\in U^{lc}$ a.e.~in $\Om$: this is the content of Theorem \ref{aperti}.
Finally, the relatively open set  $U$ is replaced by a set $K_0$ 
satisfying the in-approximation property (see Theorem \ref{generici}). 
This last step requires another iteration process. 

The following proposition, whose proof is postponed at the end of this section, 
allows us to extend Theorem \ref{generici} to the case where the boundary data
$v$ is of class $C^{1,\alpha}(\overline\Om;\R^N)$ for some $\alpha\in(0,1)$, and satisfies (\ref{cond_in_U_1}).

\begin{prop}\label{perHol}
Let $u\in C^{1,\alpha}(\overline\Om;\R^N)$ be such that
$\divv u=0$ in $\Om$.

For every $\delta>0$ there exists a piecewise affine Lipschitz map $u_{\delta}:\Om\to\R^N$
such that
\begin{align*}
&\divv u_{\delta}=0\quad\mbox{a.e.~in }\Om,\\
&u_{\delta}=u\quad\mbox{on }\pa\Om,\\
&||u_{\delta}-u||_{1,\infty}\leq\delta.
\end{align*}
\end{prop}

The following lemma represents the first step of the process leading to the proof
of Theorem~\ref{generici}.

\begin{lemma}\label{Pompe}
Let $A$, $B\in\MNN_0$ be such that {\rm rank}$(A-B)=1$ and consider
\begin{equation}\label{Clambda}
C=(1-\lambda)A+\lambda B,\qquad\mbox{for some }\lambda\in(0,1).
\end{equation}
For every $\ep>0$ there exists a piecewise affine Lipschitz map $u_{\ep}:\Om\to\R^N$ such that
\begin{align}
&\na u_{\ep}\in\MNN_0\quad\mbox{a.e.~in }\Om,\label{tesi_pompe(i)}\\
&u_{\ep}(x)=Cx\quad\mbox{for every }x\in\pa\Om,\label{tesi_pompe(ii)}\\
&{\rm dist}(\na u_{\ep},[A,B])<\ep\quad\mbox{a.e.~in }\Om,\label{tesi_pompe(iii)}\\
&\displaystyle\left|\left\{x\in\Om\,:\,{\rm dist}(\na u_{\ep},\{A,B\})\geq\ep\right\}\right|\leq c|\Om|,\label{tesi_pompe(iv)}\\
&\displaystyle\sup_{x\in\Om}|u_{\ep}(x)-Cx|<\ep.\label{tesi_pompe(v)}
\end{align}
The constant $c$ appearing in \eqref{tesi_pompe(iv)} is such that $0<c<1$ 
and depends only on the dimension $N$. 
\end{lemma}

For the proof of Lemma \ref{Pompe} it is useful to construct an explicit divergence-free vector field $u$
on the equilateral triangle $T$ with vertices
\begin{equation}\label{vertici_triangolo}
V_1=\left(-1,-\frac1{\sqrt 3}\right),\quad V_2=\left(1,-\frac1{\sqrt 3}\right),\quad
V_3=\left(0,\frac2{\sqrt 3}\right).
\end{equation}
Let $V_4$, $V_5$, and $V_6$ be the middle points of the segments joining the centre $O$ of $T$ to the middle points
of $[V_2,V_3]$, $[V_3,V_1]$, and $[V_1,V_2]$, respectively. 
We divide $T$ into the triangles $T_i$, $i=1,...,7$, illustrated in Figure \ref{disegnino}.
They are such that
\begin{equation}\label{aree}
|T_1|=|T_4|=|T_6|=\frac{|T|}6,\quad|T_2|=|T_5|=|T_7|=\frac7{48}|T|,\quad|T_3|=\frac{|T|}{16}.
\end{equation}
Consider the following vectors representing displacements applied at the points $V_4$, $V_5$, $V_6$, respectively:
\begin{equation}\label{eq:diplacements_delta}
u_4^{\delta}:=\frac{\delta}2(-1,\sqrt 3),\qquad u_5^{\delta}:=-\frac{\delta}2(1,\sqrt 3),\qquad 
u_6^{\delta}:=\delta(1,0).
\end{equation}
These three vectors have been chosen in such a way they have the same length $\delta$ and
$u_4^{\delta}$, $u_5^{\delta}$, $u_6^{\delta}$ have the same direction as 
$V_3-V_2$, $V_1-V_3$, $V_2-V_1$, respectively.
Finally, we define $u$ as the piecewise affine function defined by
\begin{equation}\label{def_u_triang}
u(V_1)= u(V_2)=u(V_3)=0,\qquad u(V_i)=u_i^{\delta},\quad i=4,5,6.
\end{equation}
It is obvious that $u=0$ on $\partial T$.  Moreover,
using the following lemma it is easy to check that $\divv u=0$ a.e.~in $T$.

\begin{figure}[htbp]
\begin{center}
\includegraphics[height=6cm]{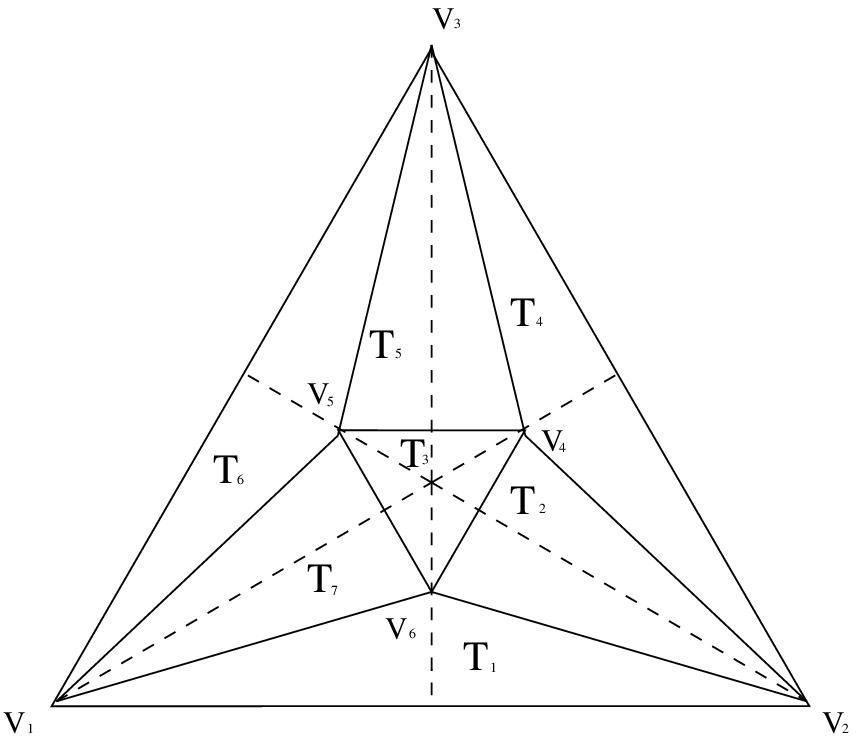}
\hspace{1.2cm}
\includegraphics[height=5.5cm]{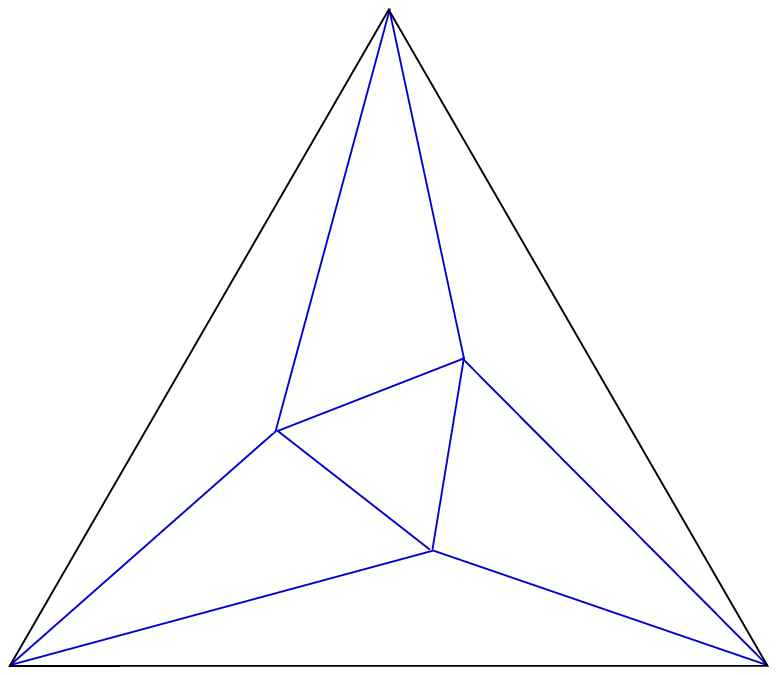}
\end{center}
\caption{Triangle $T$ and a prototype of a piecewise affine vector field $u$ such that $\divv u=0$
a.e.~in $T$ and $u=0$ on $\pa T$, used in the proof of Lemma \ref{Pompe}. 
In the first picture the triangle is underformed, in the second picture
the triangle is deformed via the displacement vector field $u$.
} 
\label{disegnino}
\end{figure}

\begin{lemma}\label{div_T}
Consider a triangle $T\subset\R^2$ with vertices $V_1$, $V_2$, and $V_3$,
and an affine function $u:T\to\R^2$ such that $u(V_1)=u(V_2)=0$.
Then, 
\begin{equation*}
\divv u=0\quad\mbox{if and only if}\quad u(V_3)\mbox{ is parallel to }V_1-V_2. 
\end{equation*}
\end{lemma}

\begin{proof}
Suppose for simplicity that $V_1-V_2$ is parallel to the first vector of the canonical basis of $\R^2$. 
Let $\nu_1$, $\nu_2$, and $\nu_3$ be the outer unit normals on the sides $[V_1,V_2]$, 
$[V_2,V_3]$, and $[V_3,V_1]$, respectively, so that 
\begin{equation}\label{pre_Pompe1}
\nu_1=(0,a),\qquad\mbox{with }a=1\mbox{ or }a=-1.
\end{equation}
Since $\na u$ is constant, from the Divergence Theorem we infer that
\begin{equation}\label{pre_Pompe2}
2|T|\tr\na u=u(V_3)\cdot(\nu_2|V_2-V_3|+\nu_3|V_3-V_1|).
\end{equation}
Using the equivalence
$
\nu_1|V_1-V_2|+\nu_2|V_2-V_3|+\nu_3|V_3-V_1|=0,
$
we obtain from (\ref{pre_Pompe2}) that
$2|T|\tr\na u(x)=-|V_1-V_2|u(V_3)\cdot\nu_1$. In view of (\ref{pre_Pompe1}), 
this implies that $\divv u=0$ if and only if the second component of $u(V_3)$ is zero.
\end{proof}
 
\begin{proof}[Proof of  Lemma \ref{Pompe}]
We follow \cite{Po} and provide the explicit proof in the case $N=2$ for the readers' convenience.

Here, we use the notation $(x,y)$ or $(\xi,\eta)$ for a point of $\R^2$,
and we consider $\MM$ endowed with the $l_{\infty}$
norm denoted by $|\cdot|_{\infty}$. It is not restrictive to suppose that
${\rm dist}$ is the distance corresponding to this norm.
The proof is divided into three cases.

\noindent{\bf Case 1.} Consider the matrix 
\begin{equation}\label{eq:matrix_E}
E:=
\left(
\begin{array}{cc}
0 & 1 \\
0 & 0
\end{array}
\right),
\end{equation} 
and suppose that $A-B=E$ and $C=0$. In this case, equation (\ref{Clambda}) gives that
$A=\lambda E$ and $B=(\lambda-1)E$, and
\begin{equation*}
{\rm dist}(M,[A,B])=\min_{0\leq\mu\leq1}|M+(\mu-\lambda)E|_{\infty},\quad\mbox{for every }M\in\MM.
\end{equation*}
From the definition of $E$ we have in particular that
\begin{equation}\label{dist_seg}
{\rm dist}(M,[A,B])=\max\{|M_{11}|,|M_{21}|,|M_{22}|\},\qquad\mbox{if }\ \lambda-1\leq M_{12}\leq\lambda.
\end{equation}
The idea is to construct a piecewise affine function $w_{\ep}$ which satisfies \eqref{tesi_pompe(i)}-\eqref{tesi_pompe(v)} on a
compact set $T_{\ep}$ with $|T_{\ep}|>0$, and then conclude the proof of Case 1 applying
Theorem \ref{Vitali}. 
Consider the piecewise affine function $u$ of components $(u_1,u_2)$ defined by 
(\ref{eq:diplacements_delta})-(\ref{def_u_triang}) on
the triangle $T$ with vertices (\ref{vertici_triangolo}). Computing the explicit expression of $u$ we get that 
\begin{equation*}
||\na u||_{L^{\infty}(T;\MM)}={\rm ess\,sup}_{(x,y)\in T}|\na u(x,y)|_{\infty}
     =\frac{\pa{u_1}}{\pa y}(x,y),\quad\mbox{for every }(x,y)\in T_1,
\end{equation*}
where $\frac{\pa{u_1}}{\pa y}(x,y)=2\sqrt 3\delta$ for every $(x,y)\in T_1$.
Choosing $\delta=\frac{\ep^3}{2\sqrt 3}$ and relabelling $u$ by $u^{\ep}$, we obtain that 
\begin{equation}\label{suTuno}
\frac{\pa u_1^{\ep}}{\pa y}=\ep^3\quad\mbox{on}\quad T_1,
\end{equation}
and that
\begin{equation}\label{uminep}
||\na u^{\ep}||_{L^{\infty}(T;\MM)}=\ep^3,\qquad\quad
||u^{\ep}||_{L^{\infty}(T;\R^2)}\leq\hat c\,\ep^3,
\end{equation}
for some constant $\hat c>0$ independent of $\ep$.
Direct computations show that
\begin{equation}\label{per_dist_seg}
\sup_T\frac{\pa u_1^{\ep}}{\pa y}=
\sup\left\{\left|\frac{\pa u_1^{\ep}}{\pa y}\right|\,:\,\frac{\pa u_1^{\ep}}{\pa y}\leq0\right\}=\ep^3.
\end{equation}
Seting $m_\ep:=\ep^3\max\{1/\lambda,1/(1-\lambda)\}$ and choosing
$
\ep^3<\min\{\lambda,1-\lambda\},
$
we have that
\begin{equation}\label{aminuno}
0<m_{\ep}<1.
\end{equation}
Then, define
\begin{equation*}
S_{\ep}:=
\left(
\begin{array}{cc}
\sqrt{m_{\ep}} & 0 \\
0 & \frac1{\sqrt{m_{\ep}}}
\end{array}
\right)\ \ \ \mbox{ and }\ \ \ \ T_{\ep}:=S_{\ep}^{-1}(T),
\end{equation*}
and note that the function
\begin{equation*}
w^{\ep}(\xi,\eta):=S_{\ep}^{-1}u^{\ep}\left(S_{\ep}\left({\xi\atop\eta}\right)\right),
\qquad\mbox{for every }(\xi,\eta)\in T_{\ep},
\end{equation*}
satisfies conditions \eqref{tesi_pompe(i)}-\eqref{tesi_pompe(v)}. Indeed, the construction of
$u^{\ep}$ implies that $\divv w^{\ep}=0$ a.e.~on $T_{\ep}$ and $w^{\ep}=0$ on $\pa T_{\ep}$.
For what concerns property \eqref{tesi_pompe(iii)}, note that
\begin{equation*}
\na w^{\ep}(\xi,\eta)=
\left(
\begin{array}{cc}
\frac{\pa u_1^{\ep}}{\pa x} & \frac1{m_{\ep}}\frac{\pa u_1^{\ep}}{\pa y}\\
m_{\ep}\frac{\pa u_2^{\ep}}{\pa x} & \frac{\pa u_2^{\ep}}{\pa y}
\end{array}
\right)_{|\left(\sqrt{m_{\ep}}\xi,\frac{\eta}{\sqrt{m_{\ep}}}\right)},\qquad\mbox{for every }(\xi,\eta)\in T_{\ep},
\end{equation*}
so that
\begin{equation}\label{partminep}
\left|\frac{\pa w^{\ep}_1}{\pa\xi}\right|, \left|\frac{\pa w^{\ep}_2}{\pa\xi}\right|, 
\left|\frac{\pa w^{\ep}_2}{\pa\eta}\right|<\ep,
\end{equation}
in view of (\ref{uminep}) and (\ref{aminuno}). Moreover, (\ref{per_dist_seg}) and the definition of $m_{\ep}$
give that $\lambda-1\leq\frac{\pa w^{\ep}_1}{\pa\eta}\leq\lambda$. This fact, together with  
(\ref{dist_seg}) and (\ref{partminep}), implies that \eqref{tesi_pompe(iii)} is true for $w^{\ep}$ a.e.~in $T_{\ep}$. 
Also, equivalence (\ref{suTuno}) gives that, for every $(\xi,\eta)\in S_{\ep}^{-1}(T_1)\subseteq T_{\ep}$,
${\rm dist}(\na w^{\ep}(\xi,\eta),\{A,B\})<\ep$, and in turn that 
\begin{equation*}
|\{(\xi,\eta)\in T_{\ep}\,:\,{\rm dist}(\na w^{\ep}(\xi,\eta),\{A,B\})\geq\ep\}|\leq|T_{\ep}\setminus S_{\ep}^{-1}(T_1)|
=\frac56|T_{\ep}|,
\end{equation*}
where the last equality is due to (\ref{aree}) and to the fact that $S_{\ep}^{-1}$ is volume-preserving.
This proves \eqref{tesi_pompe(iv)}.
From the definition of $w^{\ep}$ and from \eqref{uminep}, we infer that
\begin{equation*}
||w^{\ep}||_{L^{\infty}(T_{\ep};\R^2)}\leq
\frac{||u^{\ep}||_{L^{\infty}(T;\R^2)}}{\sqrt{m_{\ep}}}\leq\frac{\ep^{\frac32}\hat c}{\max\{\lambda,1-\lambda\}},
\end{equation*}
so that $||w_{\ep}||_{L^{\infty}(T_{\ep};\R^2)}<\ep$, if $\ep$ is sufficiently small,
and property (\ref{tesi_pompe(v)}) follows.

\noindent We remark that the function
$(\xi,\eta)\mapsto\lambda w^{\ep}(\xi/\lambda,\eta/\lambda)$ satisfies \eqref{tesi_pompe(i)}-\eqref{tesi_pompe(v)} on 
the dilated set $\lambda T_{\ep}$
for every $\lambda>0$, and the same holds for the function $(\xi,\eta)\mapsto w^{\ep}(\xi-\hat\xi,\eta-\hat\eta)$ 
on the translated set $T_{\ep}+(\hat\xi,\hat\eta)$. 
By Theorem \ref{Vitali}, there exists a disjoint numerable union 
$\bigcup_i\mathcal T_{\ep}^i\subseteq\Om$ of dilated and translated sets of $T_{\ep}$ such that
\begin{equation*}
\left|\Om\setminus\bigcup_i\mathcal T_{\ep}^i\right|=0,
\end{equation*}
and from the previous remark there exist piecewise affine Lipschitz maps 
$w^{\ep}_i:\mathcal T_{\ep}^i\to\R^2$ satisfying \eqref{tesi_pompe(i)}-\eqref{tesi_pompe(v)}
on $\mathcal T_{\ep}^i$.
Therefore, the function $w_{\ep}:\Om\to\R^2$ defined as $w_{\ep}=w^{\ep}_i$ on $\mathcal T_{\ep}^i$
for each $i$ satisfies \eqref{tesi_pompe(i)}-\eqref{tesi_pompe(v)} on $\Om$. 

\noindent{\bf Case 2.} Here, suppose $C=0$. Since $A$ and $B$ are rank-one connected, 
$0$ is an eigenvalue of $A-B$ which may have algebraic multiplicity
equal to either $1$ or $2$. The Jordan Decomposition Theorem tells us that, in the first case, there exists
an invertible matrix $L$ and $\mu\in\R\setminus\{0\}$ such that 
$A-B=L^{-1}\left(
\begin{array}{cc}
\mu & 0\\
0 & 0
\end{array}
\right)L$.
But this is impossible, because ${\rm tr}(A-B)=0$. Therefore, we have that
$A-B=L^{-1}EL$, where $E$ is defined in (\ref{eq:matrix_E}), for some invertible matrix $L$. 
Let $w$ be given by Case 1 and satisfying 
conditions \eqref{tesi_pompe(i)}-\eqref{tesi_pompe(v)},
on a rectangle $R$, for $\hat A:=LAL^{-1}$ and $\hat B:=LBL^{-1}$ (note that $\hat A-\hat B=E$ and
$(1-\lambda)\hat A+\lambda\hat B=0$). It is easy to verify that 
$u(\xi,\eta):=L^{-1}\left(w\left(L\left(\xi\atop\eta\right)\right)\right)$ satisfies 
conditions \eqref{tesi_pompe(i)}-\eqref{tesi_pompe(v)} on $L^{-1}(R)$.
Using again Theorem \ref{Vitali} and covering $\Om$ by dilated and translated copies of $L^{-1}(R)$,
we obtain a function satisfying conditions \eqref{tesi_pompe(i)}--\eqref{tesi_pompe(v)} on $\Om$.

\noindent{\bf Case 3.}
Finally, suppose $C$, $A$, and $B$ to be generic and satisfying the the hypotheses. The matrices $\hat A:=A-C$ and $\hat B:=B-C$
are such that $(1-\lambda)\hat A+\lambda\hat B=0$. Thus, from Case 2, there exists $w:\Om\to\R^2$
piecewise affine and Lipschitz satisfying \eqref{tesi_pompe(i)}-\eqref{tesi_pompe(v)} 
with $\hat A$, $\hat B$ and $0$ in place of $A$, $B$ and $C$, respectively. 
Then $u(x,y):=w(x,y)+C\left({x\atop y}\right)$ satisfies \eqref{tesi_pompe(i)}-\eqref{tesi_pompe(v)} on $\Om$.
\end{proof}

Before stating the next lemma, let us remark
that if $U$ is relatively open in $\MNN_0$, then $U^{lc}$ is relatively open in $\MNN_0$ too. 
Indeed, suppose that $U$ is relatively open in $\MNN_0$, consider $C\in U^{(1)}$, and suppose that $C+D\in\MNN_0$. 
We have that $C=(1-\lambda)A+\lambda B$ for some $0\leq\lambda\leq1$
and some $A$, $B\in U$. 
Note that $A+D$, $B+D\in\MNN_0$ and that $A+D$, $B+D\in U$ if $|D|$ is sufficiently small. 
Therefore, $C+D\in U^{(1)}$ if $|D|$ is sufficiently small, because
\begin{equation*}
C+D=(1-\lambda)(A+D)+\lambda(B+D),
\end{equation*}
and ${\rm rank}[(A+D)-(B+D)]=1$. By induction we have that $U^{lc}$ is relatively open in $\MNN_0$.

\begin{lemma}\label{4.1MuSv}
Let $U\subset\MNN_0$ be bounded and open in $\MNN_0$, and let $C\in U^{lc}$.
For every $\ep>0$ there exists a piecewise affine Lipschitz map $u_{\ep}:\Om\to\R^N$ such that
\begin{align}
&\na u_{\ep}\in U^{lc}\quad\mbox{a.e.~in }\Om,\label{tesi:4.1(i)}\\
&u_{\ep}(x)=Cx\quad\mbox{for every }x\in\pa\Om,\label{tesi:4.1(ii)}\\
&\displaystyle|\{x\in\Om\,:\,\na u_{\ep}(x)\notin U\}|<\ep|\Om|,\label{tesi:4.1(iii)}\\
&\displaystyle\sup_{x\in\Om}|u_{\ep}(x)-Cx|<\ep.\label{tesi:4.1(iv)}
\end{align}
\end{lemma}

From the proof of this lemma it is clear that the fact that $U^{lc}$ is open in $\MNN_0$ is a 
key condition to obtain the result. At a later stage, this condition is
replaced by the requirement that $U$ admits a suitable approximation $\{U_i\}$ by sets $U_i$ open in $\MNN_0$.

\begin{proof}
Suppose first that $C\in U^{(1)}$, where $U^{(1)}$ is the set of first order laminates of $U$,
and let us prove that properties (\ref{tesi:4.1(i)})-(\ref{tesi:4.1(iv)})
with $U^{(1)}$ in place of $U^{lc}$ hold for a certain $u_{\ep}:\Om\to\R^N$ piecewise affine and Lipschitz.

\noindent Consider the nontrivial case 
$C=(1-\lambda)A+\lambda B$ for some $0<\lambda<1$, $A$, $B\in U$. 
Given $\ep>0$, by Lemma \ref{Pompe} there exists a piecewise affine Lipschitz map 
$w_{\ep}^{(1)}:\Om\to\R^N$ satisfying conditions \eqref{tesi_pompe(i)}-\eqref{tesi_pompe(v)} with $\ep/2$. 
In particular, $w_{\ep}^{(1)}$ fulfils \eqref{tesi:4.1(ii)}, it is such that 
\begin{equation}\label{***}
\sup_{x\in\Om}|w^{(1)}(x)-Cx|<\frac{\ep}2,
\end{equation} 
and satisfies property \eqref{tesi:4.1(i)} with $U^{(1)}$ in place of $U^{lc}$, in view of
\eqref{tesi_pompe(i)}, \eqref{tesi_pompe(iii)}, and of the openness of $U^{(1)}$ in $\MNN_0$. 
Also, note that 
$\{w_{\ep}^{(1)}\notin U\}\subseteq\{{\rm dist}(\na w_{\ep}^{(1)},\{A,B\})\geq\ep\}$,  
again in view of the openness of $U$,
so that \eqref{tesi_pompe(iv)} implies 
\begin{equation}\label{cond_interm_set}
|\{x\in\Om\,:\,\na w_{\ep}^{(1)}(x)\notin U\}|\leq c|\Om|,
\end{equation}  
where $0<c<1$ is a constant depending only on the dimension $N$.
Building on $w_{\ep}^{(1)}$, the next part of the proof consists in an iterative process
which improves (\ref{cond_interm_set}) to (\ref{tesi:4.1(iii)}).
To simplify the notation, we write $w^{(1)}$ in place of $w_{\ep}^{(1)}$.

\noindent Since $w^{(1)}$ is piecewise affine, there exist countably many mutually
disjoint Lipschitz domains $\Om_k\subseteq\Om$ such that
$w_k^{(1)}:=w_{|\Om_k}^{(1)}$ is affine and $\left|\Om\setminus\bigcup_k\Om_k\right|=0$.
If $\left\{\Om_k^{(1)}\right\}_k\subseteq\{\Om_k\}$ are the sets where $\na w^{(1)}\notin U$,
then by (\ref{cond_interm_set})
\begin{equation}\label{somme_omj_leq}
\sum_k\left|\Om_k^{(1)}\right|
\leq c|\Om|.
\end{equation}
Applying again Lemma \ref{Pompe} on each $\Om_k^{(1)}$, with $\frac{\ep}4$ in place of $\ep$, 
we find $w_k^{(2)}:\Om_k^{(1)}\to\R^N$ piecewise affine and Lipschitz such that 
$\na w_k^{(2)}\in U^{(1)}$ a.e.~in $\Om_k^{(1)}$, $w_k^{(2)}=w^{(1)}$ on $\pa\Om_k^{(1)}$,
\begin{equation}\label{somme_omj_leq_j}
\left|\left\{x\in\Om_k^{(1)}\,:\,\na w_k^{(2)}(x)\notin U\right\}\right|\leq c|\Om_k^{(1)}|,
\end{equation}
and 
\begin{equation}\label{per_sup_ep}
\sup_{\Om_k^{(1)}}|w_k^{(2)}-w^{(1)}|<\frac{\ep}4.
\end{equation}
Defining $w^{(2)}:\Om\to\R^N$ as
\begin{equation*}
w^{(2)}:=
\left\{
\begin{array}{ll}
w^{(1)} & \mbox{on }\Om\setminus\bigcup_k\Om_k^{(1)},\\
w_k^{(2)} & \mbox{on }\Om_k^{(1)},
\end{array}
\right.
\end{equation*}
we obtain that $w^{(2)}$ is piecewise affine and Lipschitz, 
$\na w^{(2)}\in U^{(1)}$ a.e.~in $\Om$, and $w^{(2)}(x)=Cx$ for every $x\in\pa\Om$.
Also, from (\ref{***}) and (\ref{somme_omj_leq})-(\ref{per_sup_ep}) we get
\begin{equation*}
\left|\left\{x\in\Om\,:\,\na w^{(2)}(x)\notin U\right\}\right|=
\sum_k\left|\left\{x\in\Om_k^{(1)}\,:\,\na w_k^{(2)}(x)\notin U\right\}\right|
\leq c^2|\Om|,
\end{equation*}
and
\begin{equation*}
\sup_{x\in\Om}\left|w^{(2)}(x)-Cx\right|\leq
          \sup_{x\in\Om}\left\{\left|w^{(2)}(x)-w^{(1)}(x)\right|
                  +\left|w^{(1)}(x)-Cx\right|\right\}<\frac{\ep}2\left(1+\frac12\right).
\end{equation*}
Iterating this procedure gives that for every $m\in\N\setminus\{0\}$ there exists
a piecewise affine Lipschitz map
$w^{(m)}:\Om\to\R^N$ such that $\na w^{(m)}\in U^{(1)}$ a.e.~in $\Om$, $w^{(m)}(x)=Cx$ 
for every $x\in\pa\Om$, and  
\[
\left|\left\{x\in\Om\,:\,\na w^{(m)}(x)\notin U\right\}\right|\leq c^m|\Om|,
\qquad
\sup_{x\in\Om}|w^{(m)}(x)-Cx|
<\frac{\ep}2\sum_{i=0}^{m-1}\frac1{2^i}.
\]
Since $0<c<1$, then $c^m<\ep$ for $m$ sufficiently large. Setting $u_{\ep}:=w^{(m)}$ for such a big $m$,
we have obtained that $u_{\ep}$ satisfies \eqref{tesi:4.1(i)}-\eqref{tesi:4.1(iv)} with $U^{(1)}$
in place of $U^{lc}$. 

\noindent The proof of the lemma can be concluded
by a simple inductive argument which proves that if $C\in U^{(i)}$, where 
$C^{(i)}$ is the set of $i$-th order laminates of $U$, then there exists a piecewise affine Lipschitz
function satisfying \eqref{tesi:4.1(i)}-\eqref{tesi:4.1(iv)} with $U^{(i)}$
in place of $U^{lc}$. 
\end{proof}

By the same iterative method used in the proof of Lemma \ref{4.1MuSv} one can remove step by step
the set where $\na u\notin U$ obtaining the following theorem.

\begin{theorem}\label{aperti}
Let $U\subset\MNN_0$ be bounded and open in $\MNN_0$, and
suppose that $v:\Om\to\R^N$ is piecewise affine Lipschitz map such that
\begin{equation*}
\na v\in U^{lc}\quad\mbox{a.e.~in }\Om.
\end{equation*}
For every $\ep>0$ there exists a piecewise affine Lipschitz map $u_{\ep}:\Om\to\R^N$ such that
\begin{align}
&\na u_{\ep}\in U\quad\mbox{a.e.~in }\Om,\label{tesi:aperti(i)}\\
&u_{\ep}=v\quad\mbox{on }\pa\Om,\label{tesi:aperti(ii)}\\
&||u_{\ep}-v||_{\infty}<\ep.\label{tesi:aperti(iii)}
\end{align}
\end{theorem}

\begin{proof}
Consider first the case where $v$ is affine, so that $\na v(x)=Cx$ for every $x\in\Om$, 
for some $C\in U^{lc}$. Fixed $\ep>0$, by Lemma \ref{4.1MuSv} there exists a Lipschitz map 
$u^{(1)}:\Om\to\R^N$ such that $\na u^{(1)}\in U^{lc}$ a.e.~in $\Om$, $u^{(1)}=v$ on $\pa\Om$,
and such that $u_i^{(1)}:=u^{(1)}_{|\Om_i}$ is affine on countably many mutually
disjoint Lipschitz domains $\Om_i\subseteq\Om$ with $\left|\Om\setminus\bigcup_i\Om_i\right|=0$.
Note that we can write 
$\displaystyle
\Om=\bigcup_{i\in\mathcal A^{(1)}}\Om_i^{(1)}\cup \bigcup_{i\in\mathcal B^{(1)}}\Om_i^{(1)}\cup N^{(1)},
$
where 
$$
\mathcal A^{(1)}:=\left\{i\in\N\,:\,\na u_i^{(1)}\in U\right\},\quad
\mathcal B^{(1)}:=\left\{i\in\N\,:\,\na u_i^{(1)}\notin U\right\},\quad
|N^{(1)}|=0.
$$
Moreover, $u^{(1)}$ can be chosen in such a way that  
\begin{equation}\label{somme_omj_leq_00}
|M^{(1)}|
<\ep|\Om|,\qquad\quad
||u^{(1)}-v||_{\infty}<\frac{\ep}2,
\end{equation}
where $M_1:=\bigcup_{i\in\mathcal B^{(1)}}\Om_i^{(1)}$.
Applying again Lemma \ref{4.1MuSv} on each $\Om_i^{(1)}$ with $i\in\mathcal B^{(1)}$,
with $\frac{\ep}4$ in place of $\ep$, we find $u_i^{(2)}:\Om_i^{(1)}\to\R^N$
piecewise affine and Lipschitz such that $\na u_i^{(2)}\in U^{lc}$, $u_i^{(2)}=u^{(1)}$ on $\pa\Om_i^{(1)}$,
and
\begin{equation}\label{per_sup_ep_00}
||u_i^{(2)}-u^{(1)}||_{L^{\infty}(\Om_i^{(1)};\R^2)}<\frac{\ep}4,
\qquad\quad
\{x\in\Om_i^{(1)}:\na u_i^{(2)}(x)\notin U\}\leq\ep|\Om_i^{(1)}|,
\end{equation}
for every $i\in\mathcal B^{(1)}$. 
Now, define $u^{(2)}:\Om\to\R^N$ by
\begin{equation*}
u^{(2)}=
\left\{
\begin{array}{ll}
u^{(1)} & \mbox{on }\displaystyle\bigcup_{i\in\mathcal A^{(1)}}\Om_i^{(1)}\cup N^{(1)},\\
u_i^{(2)} & \mbox{on }\Om_i^{(1)},\ i\in\mathcal B^{(1)}.
\end{array}
\right.
\end{equation*} 
Again we can write 
$\displaystyle M^{(1)}=\bigcup_{i\in\mathcal A^{(2)}}\Om_i^{(2)}\cup \bigcup_{i\in\mathcal B^{(2)}}\Om_i^{(2)}\cup N^{(2)}$,
where $u_i^{(2)}$ is affine on each $\Om_i^{(2)}$ and
$$
\mathcal A^{(2)}:=\left\{i\in\N\,:\,\na u_i^{(2)}\in U\right\},\quad
\mathcal B^{(2)}:=\left\{i\in\N\,:\,\na u_i^{(2)}\notin U\right\},\quad
|N^{(2)}|=0.
$$
Setting $M^{(2)}:=\bigcup_{i\in\mathcal B^{(2)}}\Om_i^{(2)}$, we obtain that
\begin{equation}\label{somme_omj_leq_j_00}
|M^{(2)}|=|\{x\in M^{(1)}\,:\,\na u^{(2)}\notin U\}|\leq\ep|M^{(1)}|\leq\ep^2|\Om|,
\end{equation}
that $u^{(2)}$ is a piecewise affine Lipschitz function such that
$\na u^{(2)}\in U^{lc}$ a.e.~in $\Om$, that $u^{(2)}=v$ on $\pa\Om$, and that
$$
||u^{(2)}-v||_{\infty}<\frac{\ep}2\left(1+\frac12\right).
$$
Note that $u^{(2)}=u^{(1)}$ on $\Om\setminus M^{(1)}$. 
By iterating this procedure, we find the piecewise affine Lipschitz function
\begin{equation*}
u^{(m)}:=
\left\{
\begin{array}{ll}
u^{(1)} & \mbox{on }\displaystyle\bigcup_{i\in\mathcal A^{(1)}}\Om_i^{(1)}\cup N^{(1)},\\
u^{(2)} & \mbox{on }\displaystyle\bigcup_{i\in\mathcal A^{(2)}}\Om_i^{(2)}\cup N^{(2)},\\
\vdots & \\
u^{(m-1)} & \mbox{on }\displaystyle\bigcup_{i\in\mathcal A^{(m-1)}}\Om_i^{(m-1)}\cup N^{(m-1)},\\
u_i^{(m)} & \mbox{on }\Om_i^{(m-1)},\ i\in\mathcal B^{(m-1)},
\end{array}
\right.
\end{equation*} 
where $M^{(m-1)}:=\bigcup_{i\in\mathcal B^{(m-1)}}\Om_i^{(m-1)}$ is such that
\begin{equation*}
|\{x\in\Om\,:\,\na u^{(m)}\notin U\}|\leq|M^{(m)}|\leq\ep^m|\Om|.
\end{equation*}
Moreover, $\{M^{(m)}\}$ is a strictly decreasing sequence of sets, $u^{(m)}=u^{(m-1)}$ on $\Om\setminus M^{(m-1)}$, and
\begin{equation*}
\na u^{(m)}\in U^{lc}\ \mbox{ a.e.~in }\ \Om,
\qquad
u^{(m)}=v\ \mbox{ on }\ \pa\Om, 
\qquad
||u^{(m)}-v||_{\infty}<\frac{\ep}2\sum_{i=0}^{m-1}\frac1{2^i}.
\end{equation*} 
From the above properties we infer that the sequence of functions $\{u^{(m)}\}$ defines in
the limit $m\to\infty$ a piecewise affine Lipschitz function on $\Om$ satisfying
(\ref{tesi:aperti(i)})-(\ref{tesi:aperti(iii)}).

\noindent To conclude the proof it remains to consider the case where $v$ is piecewise affine.
In this case, one can repeat the above argument on every domain where $v$ is affine.
\end{proof}

We are now in a position to prove Theorem \ref{generici}, where the condition that
$U\subset\MNN_0$ is open (and bounded) in $\MNN_0$ is replaced by the condition that
$K_0\subset\MNN_0$ admits an in-approximation $\{U_i\}$. 
The idea of the proof is to construct a solution of $\na u\in K_0$ by considering suitable solutions
of $\na u_i\in U_i$.

\begin{proof}[Proof of Theorem \ref{generici}]
As in the proof of Theorem \ref{aperti}, we can assume 
without loss of generality that $v$ is affine. 
Fix $\ep>0$. Since $\na v\in U_1\subseteq U_2^{lc}$, by Theorem \ref{aperti} there exists a piecewise affine Lipschitz map 
$u_2:\Om\to\R^N$ such that $\na u_2\in U_2$ a.e.~in $\Om$, $u_2=v$ on $\pa\Om$, 
and $||u_2-v||_{\infty}<\ep/2=:\ep_2$. Consider the set
\begin{equation*}
\Om_2:=\left\{x\in\Om\,:\,{\rm dist}(x,\pa\Om)>\frac12\right\},
\end{equation*}
which is nonempty up to replacing $1/2$ by some smaller positive constant,
and let $\{\rho_{\delta}\}$ be a family of mollifiers, so that there exists
$0<\delta_2\leq1/2$ such that 
\[
||\rho_{\delta_2}{\ast}\na u_2-\na u_2||_{L^1(\Om_2;\MNN)}<1/2.
\]
For $i\geq 3$, choosing $0<\delta_i\leq\min\{\delta_{i-1},1/2^i\}$ and setting $\ep_i:=\delta_i\ep_{i-1}$,
an application of Theorem \ref{aperti} at each step yields that there exists
a piecewise affine Lipschitz map $u_i:\Om\to\R^N$ such that 
\begin{equation}\label{stare_in}
\na u_i\in U_i\ \mbox{ a.e.~in }\ \Om,\qquad 
u_i=u_{i-1}\ \mbox{ on }\ \pa\Om,\qquad
||u_i-u_{i-1}||_{\infty}<\ep_i.
\end{equation}
Moreover,
\begin{equation}\label{strong_cvg_cvx_int}
||\rho_{\delta_i}{\ast}\na u_i-\na u_i||_{L^1(\Om_i,\MNN)}<\frac1{2^i}, 
\end{equation}
where $\Om_i:=\left\{x\in\Om\,:\,{\rm dist}(x,\pa\Om)>1/2^{i-1}\right\}$.
Since $\ep_i\to0$, from the third condition in (\ref{stare_in})
we deduce that $\{u_i\}$ is a Cauchy sequence in $L^{\infty}(\Om;\R^N)$. This fact, together with
the first condition in (\ref{stare_in}) and Definition \ref{inapprox_per_nonlinear} (2), implies that
$\{u\}_i$ converges uniformly on $\overline\Om$ to some $u\in W^{1,\infty}(\Om;\R^N)$.
This implies in particular that $u$ satisfies (ii) and (iii), also in view of the fact that  
\begin{equation*}
||u_i-v||_{\infty}\leq\sum_{j=3}^i||u_j-u_{j-1}||_{\infty}+||u_2-v||_{\infty}<
\frac{\ep}2\sum_{j=0}^{i-1}\frac1{2^j}<\ep.
\end{equation*}
It remains to show that $u$ satisfies condition (i).   
Since $||\na\rho_{\delta_i}||_{L^1(\Om;\R^N)}\leq\frac C{\delta_i}$ for some constant $C>0$ independent of $\delta_i$,
using again the third condition in (\ref{stare_in}) we get
\begin{eqnarray*}
||\rho_{\delta_i}{\ast}(\na u_i-\na u)||_{L^1(\Om_i;\MNN)}
    &\leq&\frac C{\delta_i}\sum_{l=i}^{+\infty}||u_l-u_{l+1}||_{\infty}
                      <\frac C{\delta_i}\sum_{l=i}^{+\infty}\delta_l\ep_{l+1}\\
    &\leq&C\sum_{l=i}^{\infty}\ep_{l+1}<2C\ep_{i+1}.
\end{eqnarray*}
From this estimate and from (\ref{strong_cvg_cvx_int}) we can deduce that
\begin{multline}\label{eqnar:disug2}
||\na u_i-\na u||_{L^1(\Om;\MNN)}
         \leq||\na u_i-\na u||_{L^1(\Om_i;\MNN)}+||\na u_i-\na u||_{L^1(\Om\setminus\Om_i;\MNN)}\\
  \leq\frac1{2^i}+2C\ep_{i-1}+||\rho_{\delta_i}{\ast}\na u-\na u||_{L^1}
                   +||\na u_i-\na u||_{L^1(\Om\setminus\Om_i;\MNN)}.
\end{multline}
Since $\delta_i,\ep_i\to0$, and since $|\Om\setminus\Om_i|\to0$ 
and $\{\na u_i\}$ is bounded in $L^{\infty}(\Om,\R^N)$, 
from (\ref{eqnar:disug2}) we obtain that $\na u_i\to\na u$ in $L^1(\Om,\MNN)$. 
In particular, we have that, up to a subsequence, $\na u_i\to\na u$ a.e.~in $\Om$ and in turn, 
by the first condition in (\ref{stare_in}) and by Definition \ref{inapprox_per_nonlinear}, that $\na u\in K_0$.
\end{proof}

\bigskip

In what follows, we denote by $B_1$ the ball $B(0,1)\subset\R^N$ and by $[\cdot]_{\alpha}$ or $[\cdot]_{\alpha,\Delta}$  
the standard seminorm in $C^{0,\alpha}(\Delta;\R^m)$, and we provide the proof of Proposition \ref{perHol}. In order to do this,
we use a procedure already used in \cite{MuSv2}, which is based on a preliminary result (Lemma \ref{corollario_hol} below).
This consists in proving that starting from a divergence-free function $u\in C^{1,\alpha}(\overline{B_1};\R^N)$ such that 
$[\na u]_{\alpha}\leq\delta$, it is possible to construct another divergence-free function $\tilde u$ which is affine
on $B_{1/2}$ and such that $u\in C^{1,\alpha}(\overline{B_1}\setminus B_{1/2};\R^N)$ and 
$\|u-\tilde u\|_{W^{1,\infty}(B_1;\R^N)}\leq C\delta$. Such a construction can be done by using \cite[Theorem 14.2]{Daco}.
This result of Dacorogna says that for $m\geq0$ and $1<\alpha<1$ there exists a constant $K=K(m,\alpha,\Om)>0$ 
with the following property:
if $f\in C^{m,\alpha}(\overline\Om)$ satisfies $\int_{\Om}f(x)dx=0$,
then there exists $L(f)\in C^{m+1,\alpha}(\overline\Om;\R^N)$ verifying
\begin{equation*}
\left\{
\begin{array}{ll}
\divv L(f)=f & \mbox{in }\Om,\\
u=0 & \mbox{on }\pa\Om,
\end{array}
\right.
\end{equation*}
and such that 
$
||u||_{C^{m+1,\alpha}}\leq K||f||_{C^{m,\alpha}}.
$

\noindent Once the intermediate result has been established,
the proof of Proposition \ref{perHol} consists roughly in filling $\Om$ by a disjoint union 
$\bigcup_{i=1}^IB(a_i,r)$ and applying the intermediate result to each ball $B(a_i,r)$,
so that we can replace $u$ by a function $\tilde u$ which is affine of
$\bigcup_{i=1}^IB(a_i,r/2)$ and endowed with the same regularity of $u$ on
$\Om\setminus\bigcup_{i=1}^IB(a_i,r/2)$. It is then possible to
repeat the same argument to $\tilde u$ on $\Om\setminus\bigcup_{i=1}^IB(a_i,r/2)$
and then iterate it. Choosing smaller and smaller radii, this iterative procedure
converges to a piecewise affine function $u_{\infty}$ such that $u_{\infty}=u$ on $\pa\Om$.   
 
\begin{lemma}\label{corollario_hol}
For every $0<\alpha<1$, there exists a constant $C=C(N,\alpha)>0$ with the following property. 
For every $\delta>0$, $a\in \R^N$, $r>0$, and every $u\in C^{1,\alpha}(\overline{B(a,r)};\R^N)$ such that
\begin{equation}\label{ipo_cor_6.5_1}
\divv u=0\quad\mbox{in }\ B(a,r)\qquad\mbox{and}\qquad r^{\alpha}[\na u]_{\alpha}\leq\delta,
\end{equation}
there exists $\tilde u\in C^0(\overline{B(a,r)};\R^N)\cap C^{1,\alpha}(\overline{B(a,r)}\setminus B(a,r/2);\R^N)$
satisfying
\begin{align}
& \divv\tilde u=0\quad\mbox{a.e.~in }B(a,r),\label{tesitilde1}\\
& \na\tilde u(x)=\na u(a)\mbox{ for every }x\in B(a,r/2)
                    \mbox{ and }\tilde u =u\mbox{ on }\pa B(a,r/2),\label{tesitilde2}\\
& r^{-1}||u-\tilde u||_{\infty}+||\na u-\na\tilde u||_{\infty}\leq C\delta.\label{stimettatilde}
\end{align}
\end{lemma}

\begin{proof}
Let us first prove the lemma in the case $a=0$, $r=1$, and $u(0)=0$.
For any $u\in C^{1,\alpha}(\overline{B_1};\R^N)$ such that
\begin{equation}\label{ipoa0}
\divv u=0\quad\mbox{in }\ B_1\qquad\mbox{and}\qquad[\na u]_{\alpha}\leq\delta,
\end{equation}
define the affine function 
$u_0(x):=\na u(0)x$, for every $x\in\overline{B_1}$,
and the interpolation
$\hat u:=\varphi u_0+(1-\varphi)u$ on $U:=\overline{B_1}\setminus B_{\frac12}$,
where $\varphi\in C_c^{\infty}(B_1)$ is a fixed cut-off function such that $\varphi\equiv1$ on $B_{1/2}$.
It is easy to see that
\begin{equation}\label{stimahol2}
||u-u_0||_{L^\infty(B_1,R^N)}\leq||\na u-\na u_0||_{L^\infty(B_1,\MNN)}
\leq2^{\alpha}[\na u]_{\alpha}.
\end{equation}
In particular, we have that 
\begin{equation}\label{stimahol3}
||u-u_0||_{C^{0,\alpha}(B_1;\R^N)}\leq C_1(\alpha,N)[\na u]_{\alpha}.
\end{equation}
Defining
\begin{equation}\label{divatu}
f:=\divv\hat u=\na\varphi\cdot(u_0-u),
\end{equation}
we have that $f\in C^{0,\alpha}(U)$ and that $\int_{U}fdx=0$.  
Thus, by Dacorogna's result and by (\ref{divatu}), there exists $L(f)\in C^{1,\alpha}(U;\R^N)$ such that 
$\divv L(f)=f$ in $U$, $L(f)=0$ on $\pa U$, and such that
\begin{equation}\label{stimaLf}
||L(f)||_{C^{1,\alpha}(U;\R^N)}\leq C_2(N,\alpha)||u-u_0||_{C^{0,\alpha}(U;\R^N)}.
\end{equation}
Now, consider the function
\begin{equation*}
\tilde u:=
\left\{
\begin{array}{ll}
u_0 & \mbox{on }\overline{B_{\frac12}},\\
\hat u-L(f) & \mbox{on }U.
\end{array}
\right.
\end{equation*}
It is clear that $\tilde u$ is a function of class $C^0(\overline{B_1};\R^N)\cap C^{1,\alpha}(U;\R^N)$
satisfying properties (\ref{tesitilde1})-(\ref{tesitilde2}) with $a=0$ and $r=1$.
To check (\ref{stimettatilde}), note that the definition of $\tilde u$ and estimates 
(\ref{stimahol2})-(\ref{stimahol3}) imply that
\begin{equation}\label{stimatilde1}
||u-\tilde u||_{W^{1,\infty}(B_1,\R^N)}\leq C_3(N,\alpha)[\na u]_{\alpha}.
\end{equation}
By using (\ref{ipoa0}), from estimate (\ref{stimatilde1}) we deduce that
(\ref{stimettatilde}) holds with $r=1$ and $C=\tilde C_3(N,\alpha)$. 

\noindent Now, let us prove the lemma for a generic ball $B(a,r)\subset\R^N$ and for every
$u\in C^{1,\alpha}(\overline{B(a,r)};\R^N)$ satisfying (\ref{ipo_cor_6.5_1}).
The function $v\in C^{1,\alpha}(\overline{B_1};\R^N)$ defined by
\begin{equation*}
v(x):=\frac{u(rx+a)-u(a)}r
\end{equation*}
is such that $v(0)=0$ and satisfies the conditions in (\ref{ipoa0}).
The previous proof shows that then there exists
$\tilde v\in C^0(\overline{B_1};\R^N)\cap C^{1,\alpha}(\overline{B_1}\setminus B_{1/2};\R^N)$
satisfying (\ref{tesitilde1})-(\ref{stimettatilde}) with $a=0$ and $r=1$. Thus, the function
\begin{equation*}
\tilde u(x):=r\tilde v\left(\frac{x-a}r\right)+u(a)
\end{equation*}
is of class $C^0(\overline{B(a,r)};\R^N)\cap C^{1,\alpha}(\overline{B(a,r)}\setminus B(a,r/2);\R^N)$
and satisfies (\ref{tesitilde1})-(\ref{stimettatilde}) with $C=C_3(N,\alpha)$. 
\end{proof}

We are now in position to prove Proposition \ref{perHol}.

\begin{proof}[Proof of Proposition \ref{perHol}]
Fix $\delta>0$. The idea of the proof is to construct a strictly decreasing sequence of open sets $\Om_k\subset\Om$ 
and a sequence of maps $u^{(k)}$ such that $\Om_0=\Om$, $u^{(0)}=u$, 
$u^{(k)}\in W^{1,\infty}(\Om;\R^N)$, and 
\begin{align}
&||u^{(k)}-u^{(k+1)}||_{1,\infty}\leq\frac{\delta}{2^{k+1}},\label{ott1}\\
&\divv u^{(k)}=0\quad\mbox{a.e.~in }\Om,\label{ott0}\\
& u^{(k)}=u\quad\mbox{on }\pa\Om,\label{ott11}\\
&\displaystyle u^{(k+1)}=u^{(k)}\quad\mbox{on}\quad 
\bigcup_{i=1}^{n_k}\overline{A_i^{(k)}}\cup N_k
               =\Om\setminus\Om_k,
                       \quad\mbox{for every }k\geq1,\label{ott2}\\
&|\Om_{k+1}|\leq\eta|\Om_k|,\label{ott3}
\end{align}
where $\eta\in(0,1)$, 
$u^{(k)}$ is affine on each $\overline{A_i^{(k)}}$, and $N_k$ is a closed set of null measure. 
This construction implies the existence of a Lipschitz map
$v:\Om\to\R^N$ such that $u^{(k)}\to v$ in $W^{1,\infty}(\Om;\R^N)$ (by (\ref{ott1})),
$\divv v=0$ a.e.~on $\Om$ (by (\ref{ott0})),
and $v=u$ on $\pa\Om$ (by (\ref{ott11})). Moreover, (\ref{ott1}) implies that
\begin{equation*}
||u-u^{(k+1)}||_{1,\infty}\leq\sum_{i=0}^k||u^{(i)}-u^{(i+1)}||_{1,\infty}
                       \leq\delta\sum_{i=0}^k\frac1{2^{i+1}}
\leq\delta,
\end{equation*}
for every $k$, and therefore
$
||u-v||_{W^{1,\infty}}\leq\delta.
$
Finally, (\ref{ott3}) implies that
$|\Om\setminus\Om_k|\to|\Om|$. Since $\Om\setminus\Om_k$
is the set where $u^{(k)}$ is piecewise affine, and $u_l=u_k$ on $\Om\setminus\Om_k$ for very $l\geq k$,
then $v$ is piecewise affine on $\Om$.
Now, let us describe the construction of the sequences $\{\Om_k\}$ and $\{u^{(k)}\}$.
Consider $\Om''\subset\subset\Om'\subset\subset\Om$ such that $|\Om''|\geq\frac12|\Om|$,
and cover $\Om''$ by a lattice of $n_1$ disjoint open cubes $C_i^{(1)}$ with half-side $r\leq1$. If $r$ is sufficiently small,
then $\bigcup_{i=1}^{n_1}C_i^{(1)}\subseteq\Om'$. Also, 
there exists a constant $M(\Om')>0$ such that
\begin{equation}\label{passeggino}
[\na u]_{\alpha,C_i^{(1)}}\leq M(\Om'),\qquad\mbox{for every }i=1,...,n_1.
\end{equation}
Let $B_i^{(1)}$ be the open ball inscribed in $C_i^{(1)}$. By (\ref{passeggino}) we have that
\begin{equation*}
r^{\alpha}[\na u]_{\alpha,B_i^{(1)}}\leq\frac{\delta}2,\qquad\mbox{if }r\mbox{ is small enough},
\end{equation*}
so that the hypotheses of Lemma \ref{corollario_hol} are satisfied by $u$ on $B_i^{(1)}$.
Hence, denoting by $A_i^{(1)}$ the open ball with the 
same centre as $B_i^{(1)}$ and with radius $r/2$, there exists 
$u_i^{(1)}\in C^0\left(\overline{B_i^{(1)}};\R^N\right)\cap 
C^{1,\alpha}\left(\overline{B_i^{(1)}}\setminus A_i^{(1)};\R^N\right)$
such that
\begin{equation*}
\divv u_i^{(1)}=0\quad\mbox{a.e.~in }\ B_i^{(1)},
\qquad
u_i^{(1)}\mbox{ is affine in }A_i^{(1)},
\qquad u _i^{(1)}=u\quad\mbox{ on }\pa B_i^{(1)},
\end{equation*}
and
\begin{equation*}
||u-u_i^{(1)}||_{W^{1,\infty}(B_i^{(1)};\R^N)}
     \leq r^{-1}||u-u_i^{(1)}||_{L^{\infty}(B_i^{(1)};\R^N)}
               +||\na u-\na u_i^{(1)}||_{L^{\infty}(B_i^{(1)};\MNN)}\leq\frac{c\delta}2,
\end{equation*}
where the constant $c>0$ depends only on $N$ and $\alpha$. Now, define
\begin{equation*}
u^{(1)}:=\left\{
\begin{array}{ll}
u_i^{(1)} & \mbox{on }B_i^{(1)},\ i=1,...,n_1,\\
u & \mbox{on }\displaystyle\Om\setminus\bigcup_{i=1}^{n_1}B_i^{(1)},
\end{array}
\right.
\qquad
\Om_1:=\Om\setminus\bigcup_{i=1}^{n_1}\left(\overline{A_i^{(1)}}\cup\pa B_i^{(1)}\right).
\end{equation*}
Note that, since the ratio between the volume of a ball and the volume of a circumscribed cube is
a constant $\lambda=\lambda(N)\in(0,1)$, we have that
\begin{equation*}
\sum_{i=1}^{n_1}\left|A_i^{(1)}\right|
       =\lambda\sum_{i=1}^{n_1}\left|C_i^{(1)}\right|\geq\lambda|\Om''|\geq\frac{\lambda}2|\Om|,
\end{equation*}
and in turn $|\Om_1|\leq\eta|\Om|$, where $0<\eta:=1-\frac{\lambda}2<1$.
From the definition of $u^{(1)}$ we deduce that $u^{(1)}$ is piecewise affine in $\Om\setminus\Om_1$, 
that $\Om\setminus\Om_1$ is a finite union of disjoint balls (up to a null set),
that $u^{(1)}\in W^{1,\infty}(\Om;\R^N)\cap C^{1,\alpha}_{loc}(\Om_1;\R^N)$, 
that $u^{(1)}=u$ on $\pa\Om$, that ${\rm div}\,u^{(1)}=0$ a.e.~in $\Om$, and that
\begin{equation*}
||u-u^{(1)}||_{1,\infty}=
\max_{i\in\{1,...,n_1\}}||u-u_i^{(1)}||_{W^{1,\infty}(B_i^{(1)};\R^N)}\leq\frac{c(\alpha)\delta}2.
\end{equation*}
Repeating the same construction on $\Om_1$ and then iterating it defines the sequences $\{\Om_k\}$ and $\{u^{(k)}\}$.
\end{proof}


\end{document}